\documentclass[12pt]{article}
\usepackage[utf8]{inputenc}
\usepackage[backend=biber,sortcites=true,doi=false,isbn=false,url=false,eprint=false,maxbibnames=99]{biblatex}
\usepackage[T1]{fontenc}
\usepackage{microtype}
\usepackage{amsmath,amsfonts,amssymb,amsthm}
\usepackage{mathrsfs,enumitem,booktabs}
\usepackage{hyperref}
\usepackage[noabbrev,capitalise]{cleveref}
\usepackage[title]{appendix}
\usepackage{verbatim}
\usepackage{xcolor}
\usepackage{listings}
\usepackage{matlab-prettifier}
\usepackage{authblk}

\AtEveryBibitem{%
	\clearfield{note}%
}

\usepackage[caption=false]{subfig}
\usepackage{tikz}
\usetikzlibrary{calc}
\usepackage{pgfplots}
\pgfplotsset{compat=newest}
\usepgfplotslibrary{groupplots}
\usetikzlibrary{colorbrewer}
\usepgfplotslibrary{colorbrewer}
\pgfplotsset{cycle list/Dark2}
\pgfkeys{/pgf/declare function={
  arccosh(\x) = ln(\x + sqrt(\x^2-1));
  Vt(\n,\t) = cos((\n+.5)*\t)/cos(.5*\t);
  pct(\n,\t) = ((-1)^\n)*Vt(\n,\t)/(2*\n+1);
  pc(\n,\l) = pc(acos(2*\l-1));
  pct2(\n,\t) = sin((\n+.5)*\t)/sin(.5*\t)/(2*\n+1);
  lambda(\p) = sin(\p/2)^2;
  pchebt(\k,\p) = sin((\k + .5)*\p)/sin(.5*\p)/(2*\k+1);
  pcheb(\k,\p) =
    ( (2*\k + 1) * \p <= 1 ) * (1 - \k*(\k+1)/6 * \p^2 * (1 - (3*\k*(\k+1)-1)/60 * \p^2))
   +( (2*\k + 1) * \p >  1 ) * pchebt(\k,\p);
  fcheb(\k,\p) =
     and(\k== 1, (2*\k + 1) * \p <= 2 ) * sqrt(1.5) * (.5 - \p^2/8 + \p^4/576)
    +and(\k > 1, (2*\k + 1) * \p <= 2 ) * sqrt(3 / (4 * \k * (\k + 1))) * (1 - (1 + 2 * \k * (\k + 1)) * \p^2 / 20)
    +         ( (2*\k + 1) * \p >  2 ) * sin((\k + .5)*\p)/(2*\k+1)/sqrt(1 - pchebt(\k,\p)^2);
  bfrac(\c,\k) =
    (\c <= 0.7 + \k*(\k+1)/(2*1.016661482628081)) * (1 - (\c - 1) * (1.016661482628081/(\k*(\k+1)) + 1/(8*\k^2)))
   +(\c >  0.7 + \k*(\k+1)/(2*1.016661482628081)) * ((1+2*\c)/(8*\c^2)/(1.016661482628081/(\k*(\k+1)) - 1/(7*\k^4)));
  cprime(\c,\k) = \c - (41/63)*bfrac(\c,\k);
  bfrac2(\c,\n) =
    (\c <  \n^2/7.97 + .512) * (1+65/(15*\n^2-33)) * (1 - 40*\c/(10*\n^2 - \n + 19))
   +(\c >= \n^2/7.97 + .512) * ((2*\c+1)/(32*\c^2 - 10)) * (\n^2 + 1.79 - 1/\n^2);
  cprime2(\c,\k) = \c - 0.6509147135 * bfrac2(\c,2*\k+1);
  optp1(\x) = 1 - 1.5*\x;
  optp2(\x) = (1-1.118033988749895*\x)*(1-3.618033988749895*\x);
  optp3(\x) = (1-1.054958132087371*\x)*(1-1.692021471630096*\x)*(1-6.850855075327144*\x);
  xdiv1px(\x) = \x / (1 + \x);
  adivbma(\a,\b) = \a / (\b - \a);
  cheb1a(\k,\x) = acos( (\k+1-2*\k*\x)/(\k-1) );
  cheb1ah(\k,\x) = arccosh( (\k+1-2*\k*\x)/(\k-1) );
  cheb1cm1(\n,\k,\x) = adivbma(cos(\n*cheb1a(\k,\x))^2, cosh(\n*cheb1ah(\k,0))^2);
  cheb1cm1h(\n,\k,\x) = adivbma(cosh(\n*cheb1ah(\k,\x))^2, cosh(\n*cheb1ah(\k,0))^2);
  cheb1d0(\n,\k) = \n*sqrt(\k)*tanh(\n*cheb1ah(\k,0));
  cheb1cnm1low(\n,\k,\a) = (\a^2) / (2*cheb1d0(\n,\k)) + cheb1cm1(\n,\k,1/(\a^2));
  cheb1cnm1lowh(\n,\k,\a) = (\a^2) / (2*cheb1d0(\n,\k)) + cheb1cm1h(\n,\k,1/(\a^2));
  cheb1cnm1high(\n,\k) = 1/(2*cheb1d0(\n,\k)) + 1/sinh(\n*cheb1ah(\k,0))^2);
  cheb1cnm1(\n,\k,\a) = max(cheb1cnm1low(\n,\k,\a), cheb1cnm1high(\n,\k));
  cheb1cnm1h(\n,\k,\a) = max(cheb1cnm1lowh(\n,\k,\a), cheb1cnm1high(\n,\k));
}}

\newcommand{\norm}[1]{\lVert #1 \rVert}
\newcommand{\ran}{\mathsf{R}}

\newcommand{\isdef}{\mathrel{\mathop:}=}

\DeclareMathOperator*{\argmin}{arg\,min}
\DeclareMathOperator{\arccosh}{arccosh}

\newtheorem{thm}{Theorem}
\newtheorem{lemma}[thm]{Lemma}

\newtheorem{cor}[thm]{Corollary}
\newtheorem{conjecture}[thm]{Conjecture}

\numberwithin{thm}{section}

\title{Optimal Polynomial Smoothers for Multigrid V-cycles}
\author[1]{James Lottes}
\affil[1]{Google Research \protect \\\small\texttt{jlottes@google.com}}

\begin{document}

\maketitle

\abstract{The idea of using polynomial methods to improve simple smoother iterations within a multigrid method for a symmetric positive definite (SPD) system is revisited. A two-level bound going back to Hackbusch is optimized by a very simple iteration, a close cousin of the Chebyshev semi-iterative method, but based on the Chebyshev polynomials of the fourth instead of first kind. A full V-cycle bound for general polynomial smoothers is derived using the V-cycle theory of McCormick. The fourth-kind Chebyshev iteration is quasi-optimal for the V-cycle bound. The optimal polynomials for the V-cycle bound can be found numerically, achieving about an 18\% lower error contraction factor bound than the fourth-kind Chebyshev iteration, asymptotically as the number of smoothing steps $k \to \infty$. Implementation of the optimized iteration is discussed, and the performance of the polynomial smoothers is illustrated with numerical examples.}

\section{Introduction}

Given a multigrid method for solving a linear system that uses simple linear iterations for smoothers, a natural question is whether Krylov (i.e., polynomial) acceleration techniques can be applied to improve the smoother iterations. General-purpose Krylov methods don't necessarily have good smoothing properties, since they are typically designed to optimize some norm of the total error and not just its ``un-smooth'' components. In particular, a method like conjugate gradients can dramatically worsen the performance of a smoother iteration on the ``high-frequency'' error components in return for a marginal improvement in the performance on the ``smooth'' error components, which should of course be left for the coarse grid correction to handle.

There are, however, polynomial methods specifically designed to be used as smoothers \cite{adams_parallel_2003,baker_multigrid_2011,kraus_polynomial_2012}. One popular strategy for symmetric positive definite (SPD) systems, suggested by Adams, Brezina, Hu, and Tuminaro \cite{adams_parallel_2003}, is to use the standard Chebyshev semi-iteration, but adjusted to target only an upper portion $[\lambda^*, \lambda_{\max}]$ of the range of eigenvalues $[\lambda_{\min}, \lambda_{\max}]$, with $\lambda^*$ an ad-hoc parameter. Another possibility, suggested by Kraus, Vassilevski, and Zikatanov \cite{kraus_polynomial_2012} is to use the best polynomial approximation to $\lambda^{-1}$ in the uniform norm, again over the interval $[\lambda^*, \lambda_{\max}]$.
We will see that we can design methods that avoid such ad-hoc parameters as $\lambda^*$, by picking polynomials that optimize multigrid convergence bounds. We will restrict our attention to SPD systems and single-step SPD smoothers (since polynomial methods tend to be most effective in the SPD setting), and examine both a two-level bound for general polynomial smoothers going back to Hackbusch \cite{hackbusch_multi-grid_1982}, as well as develop a full V-cycle convergence bound for general polynomial smoothers starting from the V-cycle theory developed by McCormick \cite{mccormick_multigrid_1985}. We will see that a very simple Chebyshev-type iteration, which avoids the need for any ad-hoc parameter, optimizes the two-level bound and is quasi-optimal for the V-cycle bound. We will also see that a small modification to the iteration, though somewhat complicated to implement, optimizes the V-cycle bound. The effectiveness (or not) of these smoother acceleration techniques are demonstrated with numerical examples.

To fix notation, let $A$ be a symmetric positive definite matrix, and consider the ``coarse-to-fine'' half V-cycle multigrid iteration $x_{n+1} = x_n + M (b - A x_n)$ determined by the iteration matrix
\[ E = I - M A = G K \]
where $G$ is an iteration matrix for the smoother, e.g., $G = (I - B A)^k$ for $k$ steps of the iteration $x_{i+1} = x_i + B(b - A x_i)$, and $K$ is the coarse grid correction
\[ K \isdef I - P M_c P^* A,\]
where the prolongation operator $P$ is an injection from a space of smaller dimension, and
where $M_c$ is defined recursively in the same way as $M$, but for the system governed by the Galerkin coarse operator
\[ A_c \isdef P^* A P, \]
and with the recursion stopping by taking $M_c = A_c^{-1}$ at the coarsest level.

We will restrict our attention primarily to bounds of $\norm{E}_A^2$, the square of the contraction factor in the energy norm of this half V-cycle. As is well-known (see, e.g., \cite{mccormick_multigrid_1985}), this suffices for the analysis of general V-cycles.
If we denote the energy inner product adjoint for an operator $T$ by $T^{\dagger}$, so that
$  A T^{\dagger} = T^* A, $
i.e., $T^{\dagger} \isdef A^{-1} T^* A$, then $E^{\dagger}$ is the iteration matrix for a ``fine-to-coarse'' half V-cycle iteration, with equal norm $\norm{E^{\dagger}}_A = \norm{E}_A$. A general V-cycle $E_V$ decomposes as some ``fine-to-coarse'' half V-cycle $(E')^\dagger$ followed by some ``coarse-to-fine'' half V-cycle $E$,
$ E_V = E (E')^{\dagger}, $
so that
\[ \norm{E_V} \le \norm{E}_A \norm{E'}_A . \]
For a symmetric V-cycle with $E' = E$ this inequality strengthens to the equality $\norm{E_V}_A = \norm{E}_A^2$ .

Instead of using the simple smoothing iteration $G = (I - BA)^k$
we will consider polynomial iterations
\[ G = p_k(BA) \]
for some polynomial $p_k(x)$ of degree $k$ with $p_k(0) = 1$. Picking $p_k$ to optimize $\norm{G}_A$ leads to Krylov acceleration methods like the Chebyshev semi-iterative method and the method of conjugate gradients. Our central concern will be instead to investigate how to choose $p_k$ to optimize the multigrid convergence rate, through achieving optimal or nearly optimal bounds on $\norm{E}_A^2$. As polynomial methods tend to work best for symmetric iterations, we will confine our attention to the case of symmetric positive definite $B$. For example, Jacobi corresponds to $B = D^{-1}$ where $D$ is the diagonal part of $A$. Furthermore, so as to remove some redundancy, we will assume $B$ is scaled such that
\[ \rho(B A) \le 1 \]
where $\rho$ denotes the spectral radius. In particular, these assumptions imply that the single step ($k=1$) smoother
\[ G = I - \omega B A \]
has contraction factor $\norm{G}_A = \rho(G) < 1$ whenever the damping parameter $\omega$ is in the range $0 < \omega < 2$.

Our results will all be expressed in terms of the approximation property
\begin{equation} \label{eq:C-def} C \isdef \norm{A^{-1} - P A_c^{-1} P^*}_{A,B}^2 \isdef \sup_{\norm{f}_B \le 1} \norm{(A^{-1} - P A_c^{-1} P^*)f}_{A}^2 . \end{equation}
Note that a different approximation property constant $C$ is associated with each level in the V-cycle. Hackbusch \cite{hackbusch_multi-grid_1982} proved that the error contraction of the V-cycle $\norm{E}_A^2$ is bounded by the largest value of
\begin{equation}\label{eq:hackbusch-v-cycle-bound} \frac{C}{C+2k} \end{equation}
across all levels for the case $G = (1 - BA)^k$, i.e., with $\omega = 1$. Using as our main tool the general V-cycle theory developed by McCormick \cite{mccormick_multigrid_1985}, we will generalize this to a bound depending on $C$ and $p_k$. We will show that using a simple Chebyshev-type iteration for the smoother, based on a three-term recurrence, results in the V-cycle error contraction $\norm{E}_A^2$ being bounded by the largest value of
\begin{equation} \label{eq:cheb-smoother-simple-bound} \frac{C}{C+\frac{4}{3}k(k+1)} \end{equation}
across all levels.

The polynomials associated with this Chebyshev-type iteration are the lesser-known Chebyshev polynomials of the \emph{third} and \emph{fourth} kinds, also known as ``airfoil'' polynomials within the aerodynamics community, where they have been extensively studied \cite{mason_chebyshev_1993,mason_chebyshev_2002,desmarais_tables_1995}. Within the multigrid community, they have found extensive use in smoothed aggregation methods, originally as the optimal polynomials for constructing the smoothing operator to apply to some tentative prolongator, but also used for the smoothing iteration itself. The first appearance traces back to the thesis of Brezina \cite{brezina_robust_1997}, with continued development in papers by Van\^{e}k, Brezina, Tezaur, and  Krızková \cite{vanek_two-level_1996,vanek_two-grid_1999,tezaur_improved_2018}, particularly \cite{vanek_nearly_2013}. The use of these polynomials for smoothing iterations outside of the context of smoothed aggregation was suggested by Adams, Brezina, Hu, and Tuminaro \cite{adams_parallel_2003}, though at the time, the authors were apparently only aware of the existence of these polynomials for specific degrees (1, 4, and 13 being the first three). As mentioned above, Adams et al.\ also suggested the use of the usual Chebyshev semi-iteration based on the polynomials of the first kind, adjusted to target the range of 
eigenvalues $[\lambda^*, \lambda_{\max}]$ instead of $[\lambda_{\min}, \lambda_{\max}]$, with $\lambda^*$ an ad-hoc parameter.

Our main contribution is to point out that there is a very simple iteration based on the three-term recurrence satisfied by the fourth-kind Chebyshev polyonimals, namely
\[ \begin{aligned}
		x_k &= x_{k-1} + z_k, \\
		z_0 &= 0, \\
		z_k &= \frac{2k-3}{2k+1} z_{k-1} + \frac{8k-4}{2k+1} \frac{1}{\rho(BA)} B (b - A x_{k-1}),
\end{aligned}\]
and to give full multi-level V-cycle convergence bounds, particularly \eqref{eq:cheb-smoother-simple-bound}, when this iteration is used as the smoother. Note that the iteration is just as simple, if not simpler, than the standard (first-kind) Chebyshev semi-iterative method, and for use as a smoother avoids the ad-hoc choice of parameter $\lambda^*$. The fourth-kind Chebyshev polynomials and associated iteration deserves more wide-spread recognition in the broader multigrid community, as a simple method for accelerating an arbitrary multigrid V-cycle for SPD systems.

\section{Fourth-kind Chebyshev Iteration}

Let us first consider the bound due to Hackbusch \cite{hackbusch_multi-grid_1982} for the two-level iteration contraction factor (taking $M_c = A_c^{-1}$),
\begin{equation} \label{eq:weak-bound}
\begin{split}
\norm{E}_A &= \norm{E^{\dagger}}_A = \norm{(I - P A_c^{-1} P^* A) G}_A \\
&\le \norm{A^{-1} - P A_c^{-1} P^*}_{A,B} \norm{AG}_{B,A} \\
&= C^{1/2} \sup_{\lambda \in \sigma(BA)} \lambda^{1/2} |p_k(\lambda)| \\
&\le C^{1/2} \sup_{0 < \lambda \le 1} \lambda^{1/2} |p_k(\lambda)|,
\end{split}
\end{equation}
where, recall, we have taken the smoothing iteration matrix to be $G = p_k(B A)$, with $p_k$ a polynomial of degree $k$ and $p_k(0) = 1$. This bound has its shortcomings: it is only for a two-level method, and it need not be smaller than one, whereas we always have $\norm{E}_A < 1$ provided $\norm{G}_A < 1$. We will be considering other bounds without these particular shortcomings later. The reason for considering the bound \eqref{eq:weak-bound} is that the problem of picking optimal polynomials for it has an elegant solution.

Optimizing the bound \eqref{eq:weak-bound} involves solving a \emph{weighted} minimax polynomial approximation problem, with weight $w(x) = x^{1/2}$. Luckily, the solution is known. The unique monic polynomials minimizing the maximum of $(1+x)^{1/2} |p(x)|$ on $[-1, 1]$ are $2^{-n}V_n(x)$, where $V_n(x)$ are the Chebyshev polynomials of the \emph{third} kind, see Mason \cite[Property 1.1]{mason_chebyshev_1993}. Up to scaling, these are Jacobi polynomials with $\alpha=-\tfrac{1}{2}$, $\beta=\tfrac{1}{2}$, but, like the Chebyshev polynomials of the first and second kinds, enjoy particularly simple forms of the various special properties of Jacobi polynomials (see \cite{mason_chebyshev_1993,mason_chebyshev_2002,desmarais_tables_1995}). For our purposes, it will be slightly more convenient to use the Chebyshev polynomials of the fourth kind $W_n(x)$, which are essentially the same: $W_n(x) = (-1)^n V_n(-x)$. They are given by
\begin{equation} \label{eq:W-defn}
W_n(x) = \frac{\sin \, (n+\tfrac{1}{2})\theta}{\sin \tfrac{1}{2}\theta}
\end{equation}
where $x = \cos \theta$. Let us map $x \in [-1, 1]$ to $\lambda \in [0, 1]$ via $x = 1-2\lambda$. The half angle formula for sine gives $\sin \tfrac{1}{2}\theta = \lambda^{1/2}$, and we see that the weighted polynomials
\begin{equation} \label{eq:eqos}
\lambda^{1/2} W_n(1 - 2\lambda) = \sin \, (n+\tfrac{1}{2})\theta
\end{equation}
equioscillate between $\pm 1$. To enforce the constraint $p_k(0) = 1$ we need to scale by the inverse of
$ W_n(1) = 2n+1, $
giving
\begin{equation}\label{eq:cheb-pk} p_k(\lambda) = \frac{1}{2k+1} W_k(1 - 2\lambda). \end{equation}
Using \eqref{eq:eqos} we see that
\begin{equation}\label{eq:cheb-weak-bound} \sup_{0 \le \lambda \le 1} \lambda^{1/2} |p_k(\lambda)| = \frac{1}{2k+1}. \end{equation}
The first few $p_k(\lambda)$ are shown in \cref{fig:p} and the corresponding weighted polynomials $\lambda^{1/2} p_k(\lambda)$, exhibiting the equioscillation property, in \cref{fig:wp}. \Cref{fig:wp2} anticipates the multilevel bound of the next section. Explicitly, the first four polynomials are
\[\begin{aligned}
  p_0(\lambda) &= 1, \\
  p_1(\lambda) &= 1 - \tfrac{4}{3} \lambda, \\
  p_2(\lambda) &= 1 - 4\lambda + \tfrac{16}{5} \lambda^2, \\
  p_3(\lambda) &= 1 - 8\lambda + 16 \lambda^2 - \tfrac{64}{7}\lambda^3 .
  \end{aligned}
\]

\begin{figure}[ptb]
  \centering
  \subfloat[][$p_k(\lambda)$]{
  	\label{fig:p}
	\begin{tikzpicture}
	\begin{axis}[
	xlabel={$\lambda$},
	ymin=-.4, ymax=1.1,
	xmin=0, xmax=1,
	grid = both,
	cycle multi list={Dark2},
	trig format plots=rad,
	width=.49\linewidth,
	every axis plot/.append style={thick}
	]
	\pgfplotsinvokeforeach{0,...,6}{
	  \addplot +[ domain=0:pi, samples=201] ({lambda(x)}, {pcheb(#1, x)});
    }
	\end{axis}
	\end{tikzpicture}
  }
  \subfloat[][$\lambda^{1/2} p_k(\lambda)$]{
  	\label{fig:wp}
	\begin{tikzpicture}
	\begin{axis}[
	xlabel={$\lambda$},
	ymin=-.4, ymax=.4,
	xmin=0, xmax=1,
	grid = both,
	cycle multi list={Dark2},
	trig format plots=rad,
	width=.49\linewidth,
	every axis plot/.append style={thick}
	]
	\pgfplotsinvokeforeach{0,...,6}{
      \addplot +[ domain=0:pi, samples=201] ({lambda(x)}, {sqrt(lambda(x))*pcheb(#1, x)});
    }
	\end{axis}
	\end{tikzpicture}
  }

	\subfloat[][$\lambda^{1/2} p_k(\lambda) / \sqrt{1 - p_k(\lambda)^2}$]{
	\label{fig:wp2}
	\begin{tikzpicture}
	\begin{axis}[
	xlabel={$\lambda$},
	ymin=-0.4, ymax=0.65, ytick distance=0.2,
	xmin=0, xmax=1,
	grid = both,
	cycle multi list={Dark2},
	trig format plots=rad,
	width=.49\linewidth,
	every axis plot/.append style={thick}
	]
	\pgfplotsset{cycle list shift=1}
	\pgfplotsinvokeforeach{1,...,6}{
      \addplot +[ domain=0:pi, samples=201] ({lambda(x)}, {fcheb(#1, x)});
    }
	\end{axis}
	\end{tikzpicture}
}

  \caption{The first few scaled and shifted Chebyshev polynomials of the fourth kind, (a) unweighted, (b) weighted by $\lambda^{1/2}$, (c) weighted by $\sqrt{\lambda/(1-p_k(\lambda)^2)}$. }
  \label{fig:p-wp}
\end{figure}
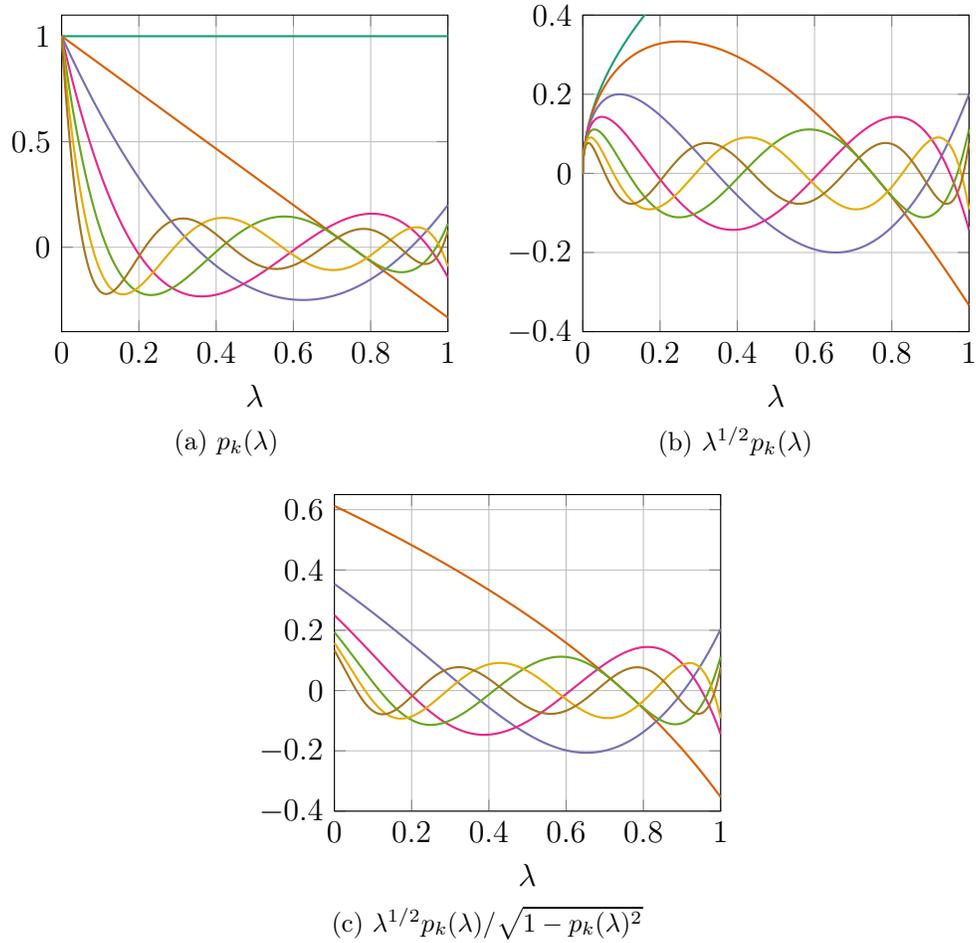

The polynomials $W_n$ satisfy the same recurrence as the Chebyshev polynomials $T_n$, $U_n$ of the first and second kinds, just with a different initial condition \cite{mason_chebyshev_1993}.
\[ W_n(x) = 2xW_{n-1}(x) - W_{n-2}(x), \qquad W_0(x) = 1, W_1(x) = 2x + 1 . \]
It follows that the scaled and shifted polynomials $p_k$ satisfy the recurrence
\[ p_k(\lambda) = \frac{4k-2}{2k+1} (1 - 2 \lambda) p_{k-1}(\lambda) - \frac{2k-3}{2k+1} p_{k-2}(\lambda). \]
We can derive an iteration based on this recurrence, exactly as is done to derive the Chebyshev semi-iterative method. The result is the iteration
\begin{equation}\label{eq:cheb-iter} \begin{aligned}
    x_k &= x_{k-1} + z_k, \\
    z_0 &= 0, \\
    z_k &= \frac{2k-3}{2k+1} z_{k-1} + \frac{8k-4}{2k+1} \frac{1}{\rho(BA)} B (b - A x_{k-1}),
\end{aligned}\end{equation}
satisfying $A^{-1} b - x_k = p_k(B A) (A^{-1} b - x_0)$. For the sake of easy reference, we have included the scaling of $B$ in \eqref{eq:cheb-iter} so that it works without our assumption $\rho(BA) \le 1$. More generally, $B$ can be scaled by any upper bound of $\rho(BA)$, with some degradation in performance. The bound
\begin{equation} \label{eq:cheb-weak-bound-alt} \norm{A p_k(B A)}_{B,A} \le \frac{1}{2k+1} \end{equation}
follows from \eqref{eq:cheb-weak-bound}, leading to the \emph{two-level} contraction factor estimate
\begin{equation} \label{eq:cheb-two-level-bound} \norm{E}_A \le \frac{C^{1/2}}{2k+1} . \end{equation}

\section{Multilevel Analysis}

We will base our multilevel analysis on the general V-cycle theory developed by McCormick \cite{mccormick_multigrid_1985}. While it is possible that an interesting analysis of polynomial smoothers could be based on the XZ identity \cite{xu_method_2002}, which considers all levels all at once, the main aim of our analysis will be to isolate the effect of the choice of polynomial from the choice of single-step smoother, so that isolating the different levels from each other is a natural first step. In retrospect, the McCormick theory can be seen as a (optimal in some sense) means of isolating the levels, making it a natural starting point for our analysis. The central idea of the theory is to decompose the solution space $\mathsf{V} = \mathbb{R}^n$ into the direct sum of the coarse space (the range of $P$) and its $A$-orthogonal complement,
\[ \mathsf{V} = \ran(P) \oplus \ran(P)^{\perp_A} . \]
Let us denote the pair of projections corresponding to this decomposition by
\[ \pi_c \isdef P A_c^{-1} P^* A, \quad\text{and}\quad \pi_f \isdef I - \pi_c. \]
Note that $\pi_f = I - P A_c^{-1} P^* A$ is the coarse grid correction with an exact coarse grid solve, and that the approximation property constant $C$ defined by \eqref{eq:C-def} can be written as $C = \norm{\pi_f A^{-1}}_{A, B}^2$. Alternately, we can write $C$ as
\begin{equation} \label{eq:C-def-alt}
	C = \sup_{u \in \ran(\pi_f) \setminus \{ 0 \}} \frac{\norm{u}_{B^{-1}}^2}{\norm{u}_{A}^2} \ge 1 .
\end{equation}
To see this, take any $A$-orthogonal basis of $\ran(P)^{\perp_A} = \ran(\pi_f)$ assembled as the columns of a matrix $Q$ so that $Q^*AQ = I$ and $\ran(Q)=\ran(\pi_f)$. Then $\pi_f = QQ^*A$ and
\[\begin{split}
	C &= \norm{\pi_f A^{-1}}_{A, B}^2 = \rho(B^{-1} A^{-1} \pi_f^* A \pi_f A^{-1}) \\
	&= \rho(B^{-1} Q Q^*)
	= \rho(Q^* B^{-1} Q) \\
	&= \sup_{v \ne 0} \frac{\norm{Q v}_{B^{-1}}^2}{\norm{v}_2^2}
	= \sup_{v \ne 0} \frac{\norm{Q v}_{B^{-1}}^2}{\norm{Q v}_A^2}
	= \sup_{u \in \ran(Q) \setminus \{0\}} \frac{\norm{u}_{B^{-1}}^2}{\norm{u}_A^2} .
\end{split}\]
The lower bound of $1$ follows from our assumption $\rho(BA) \le 1$.

The McCormick bound involves the approximation constant associated with the \emph{symmetrized smoother} $N$, defined so that the corresponding iteration matrix is
\begin{equation}\label{eq:1-NA} I - N A = G^{\dagger} G = A^{-1} G^* A G , \end{equation}
i.e., the smoother followed by its adjoint. Explicitly, $N \isdef A^{-1} - A^{-1} G^* A G A^{-1}$ and, provided $\norm{G}_A < 1$, we see that
\begin{equation}\label{eq:N<Ai} 0 < N \le A^{-1} . \end{equation}
That is, $N$ is SPD and $A^{-1} - N$ is symmetric positive semidefinite.

\begin{lemma}[McCormick V-cycle bound] \label{lem:McCormick-bound}
	If the smoother on each level is a contraction in the energy norm, $\norm{G}_A < 1$, then the symmetric V-cycle contraction factor is bounded by
	\[ \norm{E}_A^2 \le 1 - C_N^{-1} \]
	where $C_N$ is the largest value across all levels of the approximation property constant
	\[ \norm{\pi_f A^{-1}}_{A, N}^2 = \sup_{u \in \ran(\pi_f) \setminus \{ 0 \}} \frac{\norm{u}_{N^{-1}}^2}{\norm{u}_{A}^2} \ge 1 \]
	defined in terms of the symmetrized smoother $N \isdef A^{-1} - A^{-1} G^* A G A^{-1}$.
\end{lemma}

 This is essentially McCormick's Lemma~2.3 \cite{mccormick_multigrid_1985}, except that we have included the alternate expression analogous to \eqref{eq:C-def-alt} for $C_N$ (corresponding to the reciprocal of McCormick's $\beta$), which anticipates some of the ideas from McCormick's Theorem~3.4. In this case, the lower bound of $1$ for $C_N$ follows from \eqref{eq:N<Ai}. McCormick's V-cycle bound in the same form can also be found in Vassilevski \cite[p.\ 146]{vassilevski_multilevel_2008} as Theorem 5.16 together with Lemma 5.17.

It is now fairly straight-forward to prove the following V-cycle bound for our main case of interest, polynomial smoothers $G = p(BA)$. Like the two-level bound \eqref{eq:weak-bound}, this bound only depends on $C = \norm{\pi_f A^{-1}}_{A,B}^2$ and $p$. The proof is quite similar to that of Theorem~3.4 of McCormick \cite{mccormick_multigrid_1985}.

\begin{lemma} \label{lem:multilevel-poly-bound}
Let the smoother iteration (on each level) be given by
\[ G = p(BA) \]
where $B$ is symmetric positive definite and scaled such that $\rho(BA) \le 1$, and $p(x)$ is a polynomial satisfying $p(0) = 1$ and $|p(x)| < 1$ for $0 < x \le 1$, possibly different on each level. Then the V-cycle contraction factor $\norm{E}_A^2$ is bounded by the largest value across all levels of
\[ \frac{C}{C+\gamma^{-1}} \]
where $C$ is the approximation property constant $\norm{\pi_f A^{-1}}_{A, B}^2$ and
\begin{equation} \label{eq:gamma-defn} \gamma = \sup_{0 < \lambda \le 1} \frac{\lambda \, p(\lambda)^2}{1 - p(\lambda)^2} . \end{equation}
\end{lemma}
\begin{proof}
We have
\[\begin{split}
 C_N &= \sup_{u \in \ran(\pi_f) \setminus \{ 0 \}} \frac{\norm{u}_{N^{-1}}^2}{\norm{u}_{A}^2}
      = 1 + \sup_{u \in \ran(\pi_f) \setminus \{ 0 \}} \frac{\norm{u}_{N^{-1}}^2 - \norm{u}_{A}^2}{\norm{u}_{A}^2} \\
      &\le 1 + \left(\sup_{u \ne 0} \frac{\norm{u}_{N^{-1}}^2 - \norm{u}_{A}^2}{\norm{u}_{B^{-1}}^2} \right)
      \left(\sup_{u \in \ran(\pi_f) \setminus \{ 0 \}} \frac{\norm{u}_{B^{-1}}^2}{\norm{u}_{A}^2} \right) \\
      &\le 1 + \gamma C
\end{split}\]
since
\[\begin{split}
\sup_{u \ne 0} \frac{\norm{u}_{N^{-1}}^2 - \norm{u}_{A}^2}{\norm{u}_{B^{-1}}^2}
 &= \rho(B[N^{-1} - A]) \\
 &= \rho(BA [(I - p(BA)^2)^{-1} - I]) \\
 &= \sup_{\lambda \in \sigma(BA)} \lambda \, [(1 - p(\lambda)^2)^{-1} - 1] \le \gamma
\end{split}\]
where we have used that $N^{-1} = A (I - G^2)^{-1}$, which follows from \eqref{eq:1-NA} as $G^{\dagger} = G$. The result now follows by applying the previous lemma.
\end{proof}

As a quick sanity check, we note that, as a direct consequence, this lemma recovers Hackbusch's bound \eqref{eq:hackbusch-v-cycle-bound} on the contraction factor for the V-cycle with a simple iteration.

\begin{cor} \label{cor:simp-bound}
Let $B$ be SPD and scaled such that $\rho(BA) \le 1$ (on each level). If the smoother is $k$ steps of a simple damped iteration,
\[ G = (I - \omega BA)^k, \]
for possibly different $0 < \omega < 2$ and $k$ on each level,
then $\norm{E}_A^2$ is bounded by the largest value of
\[ \frac{C}{C + 2 \omega k} \]
provided
\begin{equation}\label{eq:omega-condition} (1 - \omega)^{2k} \le \frac{1}{1+2\omega k} . \end{equation}
In particular, $\omega \le \tfrac{3}{2}$ is necessary for $k=1$ and sufficient for all $k$. Asymptotically, the largest permissible $\omega$ for the bound to hold is
\[ \omega = 2 - \frac{\log 4k}{2k} + o(k^{-1}) \quad\text{as } k \to \infty . \]
\end{cor}
\begin{proof} Here $p(\lambda) = (1-\omega \lambda)^k$ and $f(\lambda) \isdef \lambda p(\lambda)^2 / (1 - p(\lambda)^2)$ is convex on $[0, 1]$, so that the supremum of $f$ is attained at either $\lambda = 1$ where $f(1)^{-1} = (1-\omega)^{-2k}-1$ or (in the limit) at $\lambda = 0$, where $\lim_{\lambda \to 0} f(\lambda)^{-1} = 2 \omega k$. The latter is smaller when \eqref{eq:omega-condition} holds. \end{proof}

\subsection{Fourth-kind Chebyshev Iteration}

The bound for the V-cycle using the fourth-kind Chebyshev iteration as smoother also follows directly from \cref{lem:multilevel-poly-bound}. As illustrated in \cref{fig:wp2}, the supremum determining $\gamma$ for the scaled and shifted Chebyshev polynomials of the fourth kind is attained in the limit $\lambda \to 0$. We prove this rigorously in \cref{sec:cheb-bound-proof}. The following bound then follows by an application of L'H\^{o}pital's rule.

\begin{cor} \label{cor:cheb-bound}
	Let $B$ be SPD and scaled such that $\rho(BA) \le 1$ (on each level).
	If \eqref{eq:cheb-iter} is used for the smoother iteration, so that
	  \[ G = p_k(BA), \]
	where $p_k$ are the scaled and shifted Chebyshev polynomials of the fourth kind given by \eqref{eq:cheb-pk}, then $\norm{E}_A^2$ is bounded by the largest value of
	\begin{equation}\label{eq:cheb-simp-bound} \frac{C}{C + \tfrac{4}{3} k (k+1) }. \end{equation}
\end{cor}
\begin{proof} See \cref{sec:cheb-bound-proof}. \end{proof}

\subsection{Optimized Polynomials}

The scaled and shifted Chebyshev polynomials of the fourth kind \eqref{eq:cheb-pk} are not optimal for the multilevel bound of \cref{lem:multilevel-poly-bound}, as they minimize the supremum of $\lambda p(\lambda)^2$ and not $\lambda p(\lambda)^2/(1-p(\lambda)^2)$. The optimal polynomials for the multilevel bound can be found numerically (see \cref{sec:app-opt-poly} for details). Explicitly, the first two are
\[\begin{aligned}
	p_1(\lambda) &= 1 - \tfrac{3}{2} \lambda, \\
	p_2(\lambda) &= (1 - \tfrac{\sqrt{5}}{2} \lambda)(1-\tfrac{5+\sqrt{5}}{2} \lambda) = 1 - (\tfrac{5}{2} + \sqrt{5})\lambda + \tfrac{5}{4}(1+\sqrt{5}) \lambda^2 .
\end{aligned}\]
The first few optimal polynomials are shown in \cref{fig:optp}, and they can be seen to exhibit the equioscillation property in \cref{fig:opt-wp}. Numerical evidence (see \cref{tbl:gamma-inv-opt}) supports the conjecture that, asymptotically, these polynomials achieve
\[
\gamma_{\text{opt}}^{-1} \sim 4\left(\frac{2k+1}{\pi}\right)^2 - \frac{2}{3} + \frac{1}{60} \left(\frac{2k+1}{\pi}\right)^{-2} + O(k^{-4}) \quad\text{as } k \to \infty,
\]
with, moreover, the first two terms of the series giving a fairly good lower bound of $\gamma_{\text{opt}}^{-1}$. Assuming this conjecture holds, we have the following V-cycle convergence bound.

\begin{conjecture} \label{conj:opt-bound}
	Let $B$ be SPD and scaled such that $\rho(BA) \le 1$ (on each level). Let the smoother on each level be given by $G = p_k(BA)$ (for possibly different $k$) where $p_k$ are the optimal polynomials for \cref{lem:multilevel-poly-bound},
	\[ p_k = \argmin_{p \in \mathbb{P}_k,\, p(0)=1} \sup_{0 < \lambda \le 1} \frac{\lambda p(\lambda)^2}{1-p(\lambda)^2}. \]
	Here $\mathbb{P}_k$ denotes the space of real polynomials of degree at most $k$. Then $\norm{E}_A^2$ is bounded by the largest value of
\begin{equation}\label{eq:opt-bound}
	\frac{C}{C+\tfrac{4}{\pi^2}(2k+1)^2-\tfrac{2}{3}}
\end{equation}
across all levels.
\end{conjecture}

As hinted at in \cref{tbl:gamma-inv-opt}, we have verified the conjecture numerically for all practical values of $k$ (into the hundreds). In the limit $k \to \infty$, the conjectured bound \eqref{eq:opt-bound} for the optimized polynomials is an improvement over the bound \eqref{eq:cheb-simp-bound} for the fourth-kind Chebyshev iteration by the factor $\frac{\pi^2}{12} \approx 0.822$, i.e., a reduction in the contraction factor bound by about 18\%.

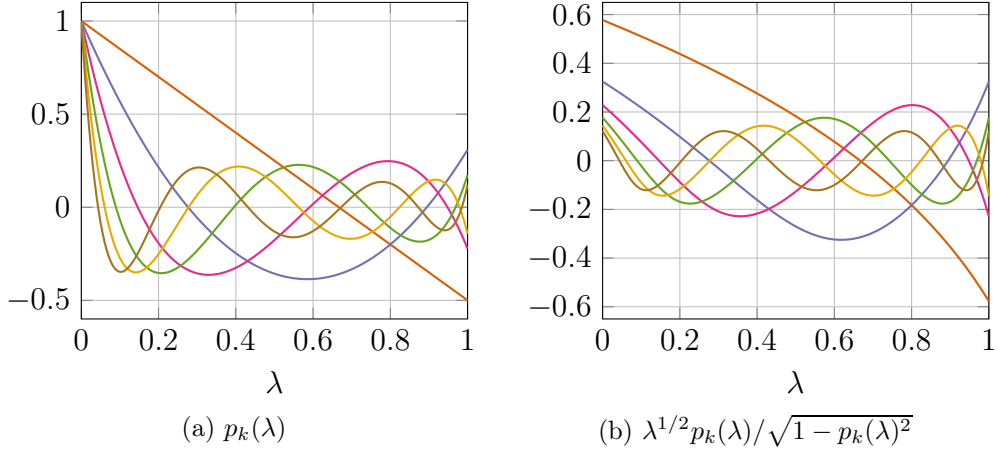
\begin{figure}[ptb]
	\centering
	\subfloat[][$p_k(\lambda)$]{
		\label{fig:optp}
		\begin{tikzpicture}
			\begin{axis}[
				xlabel={$\lambda$},
				ymin=-0.6, ymax=1.1,
				xmin=0, xmax=1,
				grid = both,
				cycle multi list={Dark2},
				trig format plots=rad,
				width=.49\linewidth,
				every axis plot/.append style={thick}
				]
				\pgfplotsset{cycle list shift=1}
				\pgfplotsinvokeforeach{1,...,6}{
					\addplot table[x=x,y=p#1] {optp.dat};
				}
			\end{axis}
		\end{tikzpicture}
	}
	\subfloat[][$\lambda^{1/2} p_k(\lambda) / \sqrt{1 - p_k(\lambda)^2}$]{
		\label{fig:opt-wp}
		\begin{tikzpicture}
			\begin{axis}[
				xlabel={$\lambda$},
				ymin=-0.65, ymax=0.65, ytick distance=0.2,
				xmin=0, xmax=1,
				grid = both,
				cycle multi list={Dark2},
				trig format plots=rad,
				width=.49\linewidth,
				every axis plot/.append style={thick}
				]
				\pgfplotsset{cycle list shift=1}
				\pgfplotsinvokeforeach{1,...,6}{
					\addplot table[x=x,y=p#1] {optpw.dat};
				}
			\end{axis}
		\end{tikzpicture}
	}
	
	\caption{The optimal polynomials $p_k$ for the bound of \cref{lem:multilevel-poly-bound}, (a)~unweighted, (b)~weighted by $\sqrt{\lambda/(1-p_k(\lambda)^2)}$. }
	\label{fig:opt-p-wp}
\end{figure}

\begin{table}[htb]
	\centering
	\begin{tabular}{rr@{.}lr@{.}lr@{.}lr@{.}l} \toprule
		$k$ & 
		\multicolumn{2}{c}{$\gamma_{\text{opt}}^{-1}$} &
		\multicolumn{2}{c}{$\frac{4}{\pi^2}(2k+1)^2 - \frac{2}{3}$} &
		\multicolumn{2}{c}{Difference} &
		\multicolumn{2}{c}{$\frac{\pi^2}{60}(2k+1)^{-2}$} \\ \midrule
		1 & 3 &     & 2&981   &  1&$91 \times 10^{-2}$ & 1&$83 \times 10^{-2}$\\
		2 & 9&4721  & 9&465   &  6&$68 \times 10^{-3}$ & 6&$58 \times 10^{-3}$\\
		3 & 19&1957 & 19&1923 &  3&$38\times 10^{-3}$ & 3&$36 \times 10^{-3}$\\
		4 & 32&1634 & 32&1614 &  2&$04\times 10^{-3}$ & 2&$03 \times 10^{-3}$\\
		5 & 48&3742 & 48&3728 &  1&$36\times 10^{-3}$ & 1&$36 \times 10^{-3}$ \\
		10 & 178&0643 & 178&0639 &  3&$73\times 10^{-4}$ & 3&$73 \times 10^{-4}$ \\
		100 & 16373&241899 & 16373&241895 &  4&$07\times 10^{-6}$ & 4&$07 \times 10^{-6}$ \\
		1000 & 1&$623\times 10^{6}$ & 1&$623\times 10^{6}$ &  4&$11\times 10^{-8}$ & 4&$11 \times 10^{-8}$ \\
		\bottomrule \end{tabular}
	\caption{Optimal $\gamma^{-1}$ for \cref{lem:multilevel-poly-bound}.} \label{tbl:gamma-inv-opt}
\end{table}

\section{Implementation Issues}

The iteration \eqref{eq:cheb-iter} is the natural implementation for a smoother using the Chebyshev polynomials of the fourth kind, derived from the three-term recurrence they satisfy. The numerically determined optimal polynomials for the V-cycle bound of \cref{lem:multilevel-poly-bound} do not satisfy a three-term recurrence, and so require a different implementation.

One obvious approach is to translate the factorization into monomials involving the roots $r_i$,
\[ p_k(\lambda) = \prod_{i=1}^k \left(1 - \frac{\lambda}{r_i}\right) \, ,\]
directly into the iteration
\[ x_{i+1} = x_i + \frac{1}{r_i} B (b - A x_i) . \]
While probably fine for small $k$, a concern of this approach is that the polynomials corresponding to the intermediate steps are poorly behaved. Depending on the chosen ordering of roots, the iteration may first dampen some error modes by several orders of magnitudes before then amplifying them by several orders of magnitude, or vice versa.

Instead, we propose using an iteration based on the expansion in the basis of fourth-kind Chebyshev polynomials
\[p_k(\lambda) = \sum_{i=0}^k \alpha_i W_i(1 - 2\lambda), \]
exploiting the fact that we expect the orthogonal (w.r.t.\ a weight function) Chebyshev polynomials of the fourth kind $W_i$ to be a good basis for representing the optimal $p_k$ and the fact that we have a nice iteration, namely \eqref{eq:cheb-iter}, for computing the basis $W_i$. The Gaussian quadrature rule associated with the polynomials $W_i$ can be used to determine the coefficients $\alpha_i$, requiring only that $p_k$ be evaluated at the quadrature nodes (see \cref{sec:app-opt-poly} for details). We propose then the iteration
\begin{equation}\label{eq:opt-iter} \begin{aligned}
		z_0 &= 0, \quad r_0 = b - A x_0, \\
		z_i &= \frac{2i-3}{2i+1} z_{i-1} + \frac{8i-4}{2i+1} \frac{1}{\rho(BA)} B r_{i-1}, \\
		x_i &= x_{i-1} + \beta_i z_i, \quad r_i = r_{i-1} - A z_i.
\end{aligned}\end{equation}
This is equivalent to \eqref{eq:cheb-iter} when we take $\beta_i = 1$. Note that $\beta_i$ does not appear in the ``residual'' update, so that it is \emph{not} the case that $r_i = b - A x_i$. Instead, $r_i$ tracks the residual of the unmodified fourth-kind Chebyshev iteration, and the vectors $z_i$ are identical in both iterations. From the properties of \eqref{eq:cheb-iter}, we know that $z_i = [\tfrac{1}{2i-1}W_{i-1}(BA) - \tfrac{1}{2i+1}W_i(BA)] e_0$ where $e_0 = A^{-1} b - x_0$. It follows that the polynomial $p_k(\lambda)$ corresponding to the iteration \eqref{eq:opt-iter}, i.e., so that $A^{-1} b - x_k = p_k(BA) e_0$, is given by
\[p_k(\lambda) = \sum_{i=0}^k \frac{\beta_i-\beta_{i+1}}{2i+1} W_i(1-2\lambda) \]
provided we let $\beta_0 = 1$ and $\beta_{k+1} = 0$. Hence, the coefficients $\beta_i$ can be found from the coefficients $\alpha_i$ via the recursion
\[ \beta_{i+1} = \beta_i - (2i+1)\alpha_i . \]
For the polynomials that optimize the multilevel bound \eqref{lem:multilevel-poly-bound}, the resulting $\beta_i$ (which depend on $k$) are in the interval $[1, 1.6)$, for $k$ up to several hundred. We can thus think of the optimized iteration \eqref{eq:opt-iter} as an ``over-relaxed'' version of the fourth-kind Chebyshev iteration \eqref{eq:cheb-iter}. Explicit Matlab code for computing the coefficients $\beta_i$ is given in \cref{sec:app-opt-poly}.

\section{Numerical Results}

We examine the effectiveness of the considered polynomial accelerations of smoothers by looking at three numerical examples, all variants of geometric multigrid applied to a 2D Poisson problem discretized by the finite element method (FEM). In each case, the problem to be solved consists of a $1024 \times 1024$ grid of rectangular bilinear finite elements, and we consider diagonal single-step smoothers (e.g., damped Jacobi). We use simple bilinear interpolation in every case and use the Galerkin construction for the coarse problems (which is identical to the direct discretization of the coarse problem). For the first example, we take the grid spacing to be uniform, the diffusion coefficient to be constant, while varying the apsect ratio $\Delta y/\Delta x$, which has the effect of varying the approximation constant $C$. For our second and third examples, we group the elements in to a grid of square ``macroelements'' each consisting of $N$ by $N$ finite elements, resulting in $1024/N$ macroelements in each direction. In the second example, we consider coefficient jumps, alternating the diffusion coefficients of the macroelements in a checkerboard pattern. For our third example, we consider instead nonuniform mesh spacing, taking the nodes within each macroelement to be given by an affine mapping of the Chebyshev nodes. This last example is closely related to low-order preconditioners for high-order elements (considered, e.g., in \cite{heys_algebraic_2005}). For all examples, we will primarily be looking at the asymptotic symmetric V-cycle error contraction factor $\norm{E}_A^2$, measured by applying the symmetric V-cycle in a generalized Lanczos iteration \cite{van_der_vorst_generalized_1982}.

\subsection{Uniform grid with varying aspect ratio}

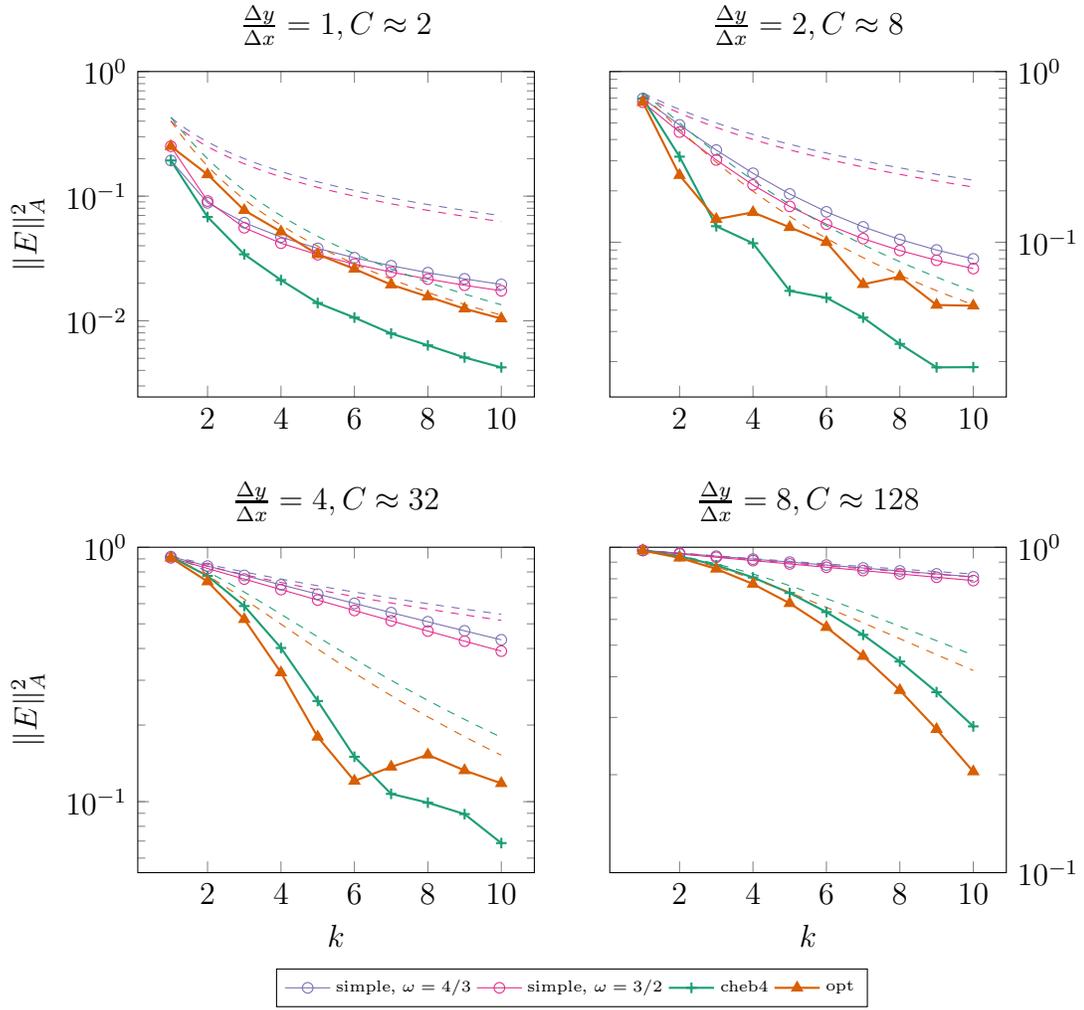
\begin{figure}[ptb]
	\centering
	\begin{tikzpicture}
		\begin{groupplot}[%
			group style={%
				group size=2 by 2,
				vertical sep=2cm
			},
			ymax=1,
			width=.5\linewidth
			]
			\nextgroupplot[title={$\tfrac{\Delta y}{\Delta x} = 1, C\approx 2$},ymode=log,ylabel={$\norm{E}_A^2$},
			legend style={font=\tiny,at={($(0,0)+(1cm,1cm)$)}, legend columns=4,fill=none,draw=black,anchor=center,align=center},
			legend to name=legend1
			]
			\pgfmathsetmacro{\aspect}{1}
			
			\addlegendimage{mark=o,Dark2-C}; \addlegendentry{simple, $\omega=4/3$};
			\addlegendimage{mark=o,Dark2-D}; \addlegendentry{simple, $\omega=3/2$};
			
			\addlegendimage{mark=+,Dark2-A,style=thick,mark options={thick}}; \addlegendentry{cheb4};
			\addlegendimage{mark=triangle*,Dark2-B,style=thick}; \addlegendentry{opt};
			
			\addplot [mark=o,Dark2-C] table[x=k,y=w43,col sep=comma] {data-10-\aspect.dat};
			\addplot [mark=o,Dark2-D] table[x=k,y=w32,col sep=comma] {data-10-\aspect.dat};
			
			\addplot [mark=+,Dark2-A,style=thick,mark options={thick}] table[x=k,y=cheb4,col sep=comma] {data-10-\aspect.dat};
			\addplot [mark=triangle*,Dark2-B,style=thick] table[x=k,y=opt,col sep=comma] {data-10-\aspect.dat};
			
			\addplot [style=dashed,Dark2-C, domain=1:10, samples=101,] {2*\aspect^2/(2*\aspect^2 + 2*(4/3)*x)};
			\addplot [style=dashed,Dark2-D, domain=1:10, samples=101,] {2*\aspect^2/(2*\aspect^2 + 2*(3/2)*x)};
			\addplot [style=dashed,Dark2-A, domain=1:10, samples=101,] {2*\aspect^2/(2*\aspect^2 + (4/3)*x*(x+1))};
			\addplot [style=dashed,Dark2-B, domain=1:10, samples=101,] {2*\aspect^2/(2*\aspect^2 + ((2/3.1415926535897932384626)*(2*x+1))^2 - 2/3 )};
			
			\coordinate (c1) at (rel axis cs:0,1);
			
			\nextgroupplot[title={$\tfrac{\Delta y}{\Delta x} = 2, C\approx 8$}, ymode=log,yticklabel pos=right]
			\pgfmathsetmacro{\aspect}{2}
			\addplot [mark=o,Dark2-C] table[x=k,y=w43,col sep=comma] {data-10-\aspect.dat};
			\addplot [mark=o,Dark2-D] table[x=k,y=w32,col sep=comma] {data-10-\aspect.dat};
			\addplot [mark=+,Dark2-A,style=thick,mark options={thick}] table[x=k,y=cheb4,col sep=comma] {data-10-\aspect.dat};
			\addplot [mark=triangle*,Dark2-B,style=thick] table[x=k,y=opt,col sep=comma] {data-10-\aspect.dat};
			
			\addplot [style=dashed,Dark2-C, domain=1:10, samples=101,] {2*\aspect^2/(2*\aspect^2 + 2*(4/3)*x)};
			\addplot [style=dashed,Dark2-D, domain=1:10, samples=101,] {2*\aspect^2/(2*\aspect^2 + 2*(3/2)*x)};
			\addplot [style=dashed,Dark2-A, domain=1:10, samples=101,] {2*\aspect^2/(2*\aspect^2 + (4/3)*x*(x+1))};
			\addplot [style=dashed,Dark2-B, domain=1:10, samples=101,] {2*\aspect^2/(2*\aspect^2 + ((2/3.1415926535897932384626)*(2*x+1))^2 - 2/3 )};
			
			\coordinate (c2) at (rel axis cs:1,1);
			
			\nextgroupplot[title={$\tfrac{\Delta y}{\Delta x} = 4, C\approx 32$}, ymode=log,ylabel={$\norm{E}_A^2$},xlabel=$k$]
			\pgfmathsetmacro{\aspect}{4}
			\addplot [mark=o,Dark2-C] table[x=k,y=w43,col sep=comma] {data-10-\aspect.dat};
			\addplot [mark=o,Dark2-D] table[x=k,y=w32,col sep=comma] {data-10-\aspect.dat};
			\addplot [mark=+,Dark2-A,style=thick,mark options={thick}] table[x=k,y=cheb4,col sep=comma] {data-10-\aspect.dat};
			\addplot [mark=triangle*,Dark2-B,style=thick] table[x=k,y=opt,col sep=comma] {data-10-\aspect.dat};
			
			\addplot [style=dashed,Dark2-C, domain=1:10, samples=101,] {2*\aspect^2/(2*\aspect^2 + 2*(4/3)*x)};
			\addplot [style=dashed,Dark2-D, domain=1:10, samples=101,] {2*\aspect^2/(2*\aspect^2 + 2*(3/2)*x)};
			\addplot [style=dashed,Dark2-A, domain=1:10, samples=101,] {2*\aspect^2/(2*\aspect^2 + (4/3)*x*(x+1))};
			\addplot [style=dashed,Dark2-B, domain=1:10, samples=101,] {2*\aspect^2/(2*\aspect^2 + ((2/3.1415926535897932384626)*(2*x+1))^2 - 2/3 )};
			
			\nextgroupplot[title={$\tfrac{\Delta y}{\Delta x} = 8, C\approx 128$}, ymode=log,xlabel=$k$,yticklabel pos=right,ymin=.1]
			\pgfmathsetmacro{\aspect}{8}
			\addplot [mark=o,Dark2-C] table[x=k,y=w43,col sep=comma] {data-10-\aspect.dat};
			\addplot [mark=o,Dark2-D] table[x=k,y=w32,col sep=comma] {data-10-\aspect.dat};
			\addplot [mark=+,Dark2-A,style=thick,mark options={thick}] table[x=k,y=cheb4,col sep=comma] {data-10-\aspect.dat};
			\addplot [mark=triangle*,Dark2-B,style=thick] table[x=k,y=opt,col sep=comma] {data-10-\aspect.dat};
			
			\addplot [style=dashed,Dark2-C, domain=1:10, samples=101,] {2*\aspect^2/(2*\aspect^2 + 2*(4/3)*x)};
			\addplot [style=dashed,Dark2-D, domain=1:10, samples=101,] {2*\aspect^2/(2*\aspect^2 + 2*(3/2)*x)};
			\addplot [style=dashed,Dark2-A, domain=1:10, samples=101,] {2*\aspect^2/(2*\aspect^2 + (4/3)*x*(x+1))};
			\addplot [style=dashed,Dark2-B, domain=1:10, samples=101,] {2*\aspect^2/(2*\aspect^2 + ((2/3.1415926535897932384626)*(2*x+1))^2 - 2/3 )};
			
		\end{groupplot}
		\coordinate (c3) at ($(c1)!.5!(c2)$);
		\node[below] at (c3 |- current bounding box.south)
		{\pgfplotslegendfromname{legend1}};	
		
	\end{tikzpicture}
	
	\caption{V-Cycle error contraction factors for Poisson discretized on a uniform grid.}
	\label{fig:nr-simp1}
\end{figure}

We first consider the case of a uniform mesh with constant diffusion coefficient and varying aspect ratio, with Dirichlet boundary conditions on each boundary.  In practice, the degradation of performance associated with high aspect ratios would probably be better addressed with a different multigrid strategy such as semi-coarsening. Here we use the aspect ratio to control the approximation property constant $C$, for a problem amenable to relatively simple analysis.

\Cref{fig:nr-simp1} illustrates the performance of the fourth-kind Chebyshev iteration ``cheb4'' \eqref{eq:cheb-iter}, the optimized iteration ``opt'', and the simple iteration with damping parameters $\omega=4/3$ and $\omega=3/2$ as smoothers in a symmetric multigrid V-cycle. The damping parameters for the simple iterations were chosen so that they coincide with the two polynomial smoothers for $k=1$ smoothing step. We use (rescaled) Jacobi as the single-step smoother, which in this case is the same as Richardson as the diagonal of $A$ is constant. The approximation constant $C$ for this problem asymptotes to twice the element aspect ratio squared as the grid size increases (as can be seen from a Fourier analysis), and the four subfigures correspond to aspect ratios of 1, 2, 4, and 8, and thus to $C\approx2,8,32,128$ on the first level. (The values of $C$ taper off on the coarser levels.) The solid lines with markers are the measured (via Lanczos) V-cycle contraction factors for different numbers of smoothing steps $k$, taken to be the same on each level. For each solid line, the dashed line of the same color is the bound derived from \cref{lem:multilevel-poly-bound}, i.e., \cref{cor:simp-bound} for the simple iteration, \cref{cor:cheb-bound} for the Chebyshev iteration, and \eqref{eq:opt-bound} for the optimized iteration.

 In all cases, the fourth-kind Chebyshev iteration significantly out-performs the simple iterations. While there is a significant gap between the measured contraction factors and the corresponding bounds, apart from the optimized iteration the measured contraction factors and bounds qualitatively follow the same trends, for the most part. The optimized iteration performs worse than even the simple iteration for the ideal case of aspect ratio $1$ for $k<5$. However, it does deliver the expected improvement over the fourth-kind Chebyshev iteration as suggested by the improved bound, for $C$ large enough and $k$ small enough.  While the performance of the optimized iteration can be relatively poor for small $C$ and/or large $k$, we observe that it approaches the bound derived from \cref{lem:multilevel-poly-bound} quite closely at some points, showing that the bound can be rather sharp in some cases.
 
  The non-monotonic behavior of some of the curves as $k$ increases is perhaps surprising. As we shall see, this is ultimately due to the fact that the polynomials $p_k(\lambda)$ are not themselves monotonic in $k$ for fixed $\lambda$ (apart from the simple iterations), together with the fact that the convergence rate is more sensitive to the values of the polynomial at some eigenvalues than at others.

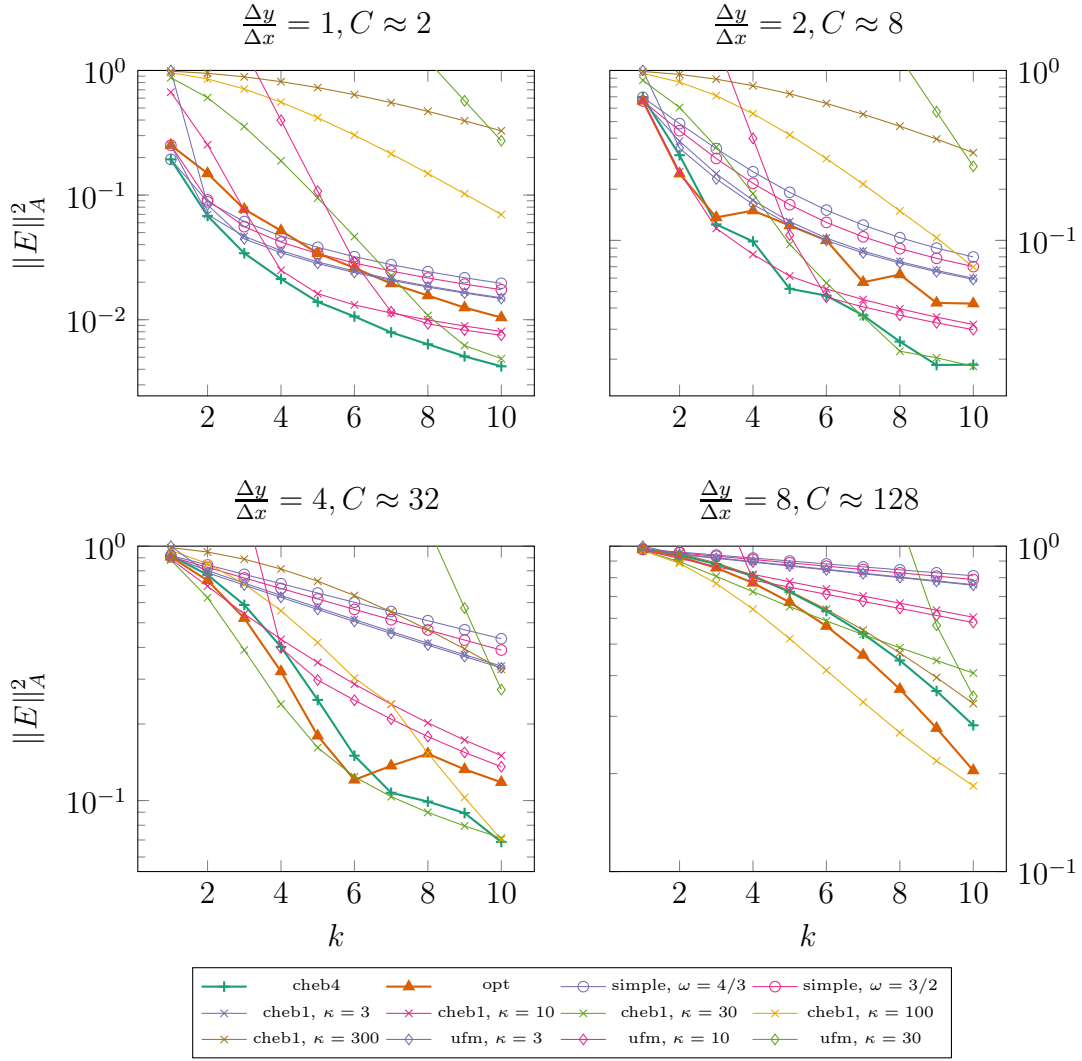
\begin{figure}[ptb]
	\centering
	\begin{tikzpicture}
		\begin{groupplot}[%
			group style={%
				group size=2 by 2,
				vertical sep=2cm
			},
			ymax=1,
			width=.5\linewidth
			]
			\nextgroupplot[title={$\tfrac{\Delta y}{\Delta x} = 1, C\approx 2$},ymode=log,ylabel={$\norm{E}_A^2$},
			legend style={font=\tiny,at={($(0,0)+(1cm,1cm)$)}, legend columns=4,fill=none,draw=black,anchor=center,align=center},
			legend to name=leg
			]
			\pgfmathsetmacro{\aspect}{1}
			
			\addlegendimage{mark=+,Dark2-A,style=thick,mark options={thick}}; \addlegendentry{cheb4};
			\addlegendimage{mark=triangle*,Dark2-B,style=thick}; \addlegendentry{opt};
			
			\addlegendimage{mark=o,Dark2-C}; \addlegendentry{simple, $\omega=4/3$};
			\addlegendimage{mark=o,Dark2-D}; \addlegendentry{simple, $\omega=3/2$};
			
			\addlegendimage{mark=x,Dark2-C}; \addlegendentry{cheb1, $\kappa=3$};
			\addlegendimage{mark=x,Dark2-D}; \addlegendentry{cheb1, $\kappa=10$};
			\addlegendimage{mark=x,Dark2-E}; \addlegendentry{cheb1, $\kappa=30$};
			\addlegendimage{mark=x,Dark2-F}; \addlegendentry{cheb1, $\kappa=100$};
			\addlegendimage{mark=x,Dark2-G}; \addlegendentry{cheb1, $\kappa=300$};
			
			\addlegendimage{mark=diamond,Dark2-C}; \addlegendentry{ufm, $\kappa=3$};
			\addlegendimage{mark=diamond,Dark2-D}; \addlegendentry{ufm, $\kappa=10$};
			\addlegendimage{mark=diamond,Dark2-E}; \addlegendentry{ufm, $\kappa=30$};
			
			\addplot [mark=o,Dark2-C] table[x=k,y=w43,col sep=comma] {data-10-\aspect.dat};
			\addplot [mark=o,Dark2-D] table[x=k,y=w32,col sep=comma] {data-10-\aspect.dat};
			
			\addplot [mark=+,Dark2-A,style=thick,mark options={thick}] table[x=k,y=cheb4,col sep=comma] {data-10-\aspect.dat};
			\addplot [mark=triangle*,Dark2-B,style=thick] table[x=k,y=opt,col sep=comma] {data-10-\aspect.dat};
			\addplot [mark=x,Dark2-C] table[x=k,y=cheb1_3,col sep=comma] {data-10-\aspect.dat};
			\addplot [mark=x,Dark2-D] table[x=k,y=cheb1_10,col sep=comma] {data-10-\aspect.dat};
			\addplot [mark=x,Dark2-E] table[x=k,y=cheb1_30,col sep=comma] {data-10-\aspect.dat};
			\addplot [mark=x,Dark2-F] table[x=k,y=cheb1_100,col sep=comma] {data-10-\aspect.dat};
			\addplot [mark=x,Dark2-G] table[x=k,y=cheb1_300,col sep=comma] {data-10-\aspect.dat};
			
			\addplot [mark=diamond,Dark2-C] table[x=k,y=ufm_3,col sep=comma] {data-10-\aspect.dat};
			\addplot [mark=diamond,Dark2-D] table[x=k,y=ufm_10,col sep=comma] {data-10-\aspect.dat};
			\addplot [mark=diamond,Dark2-E] table[x=k,y=ufm_30,col sep=comma] {data-10-\aspect.dat};
			
			\coordinate (c1) at (rel axis cs:0,1);
			
			\nextgroupplot[title={$\tfrac{\Delta y}{\Delta x} = 2, C\approx 8$}, ymode=log,yticklabel pos=right]
			\pgfmathsetmacro{\aspect}{2}
			\addplot [mark=o,Dark2-C] table[x=k,y=w43,col sep=comma] {data-10-\aspect.dat};
			\addplot [mark=o,Dark2-D] table[x=k,y=w32,col sep=comma] {data-10-\aspect.dat};
			
			\addplot [mark=+,Dark2-A,style=thick,mark options={thick}] table[x=k,y=cheb4,col sep=comma] {data-10-\aspect.dat};
			\addplot [mark=triangle*,Dark2-B,style=thick] table[x=k,y=opt,col sep=comma] {data-10-\aspect.dat};
			\addplot [mark=x,Dark2-C] table[x=k,y=cheb1_3,col sep=comma] {data-10-\aspect.dat};
			\addplot [mark=x,Dark2-D] table[x=k,y=cheb1_10,col sep=comma] {data-10-\aspect.dat};
			\addplot [mark=x,Dark2-E] table[x=k,y=cheb1_30,col sep=comma] {data-10-\aspect.dat};
			\addplot [mark=x,Dark2-F] table[x=k,y=cheb1_100,col sep=comma] {data-10-\aspect.dat};
			\addplot [mark=x,Dark2-G] table[x=k,y=cheb1_300,col sep=comma] {data-10-\aspect.dat};
			
			\addplot [mark=diamond,Dark2-C] table[x=k,y=ufm_3,col sep=comma] {data-10-\aspect.dat};
			\addplot [mark=diamond,Dark2-D] table[x=k,y=ufm_10,col sep=comma] {data-10-\aspect.dat};
			\addplot [mark=diamond,Dark2-E] table[x=k,y=ufm_30,col sep=comma] {data-10-\aspect.dat};
			
			\coordinate (c2) at (rel axis cs:1,1);
			
			\nextgroupplot[title={$\tfrac{\Delta y}{\Delta x} = 4, C\approx 32$}, ymode=log,ylabel={$\norm{E}_A^2$},xlabel=$k$]
			\pgfmathsetmacro{\aspect}{4}
			\addplot [mark=o,Dark2-C] table[x=k,y=w43,col sep=comma] {data-10-\aspect.dat};
			\addplot [mark=o,Dark2-D] table[x=k,y=w32,col sep=comma] {data-10-\aspect.dat};
			
			\addplot [mark=+,Dark2-A,style=thick,mark options={thick}] table[x=k,y=cheb4,col sep=comma] {data-10-\aspect.dat};
			\addplot [mark=triangle*,Dark2-B,style=thick] table[x=k,y=opt,col sep=comma] {data-10-\aspect.dat};
			\addplot [mark=x,Dark2-C] table[x=k,y=cheb1_3,col sep=comma] {data-10-\aspect.dat};
			\addplot [mark=x,Dark2-D] table[x=k,y=cheb1_10,col sep=comma] {data-10-\aspect.dat};
			\addplot [mark=x,Dark2-E] table[x=k,y=cheb1_30,col sep=comma] {data-10-\aspect.dat};
			\addplot [mark=x,Dark2-F] table[x=k,y=cheb1_100,col sep=comma] {data-10-\aspect.dat};
			\addplot [mark=x,Dark2-G] table[x=k,y=cheb1_300,col sep=comma] {data-10-\aspect.dat};
			
			\addplot [mark=diamond,Dark2-C] table[x=k,y=ufm_3,col sep=comma] {data-10-\aspect.dat};
			\addplot [mark=diamond,Dark2-D] table[x=k,y=ufm_10,col sep=comma] {data-10-\aspect.dat};
			\addplot [mark=diamond,Dark2-E] table[x=k,y=ufm_30,col sep=comma] {data-10-\aspect.dat};
			
			\nextgroupplot[title={$\tfrac{\Delta y}{\Delta x} = 8, C\approx 128$}, ymode=log,xlabel=$k$,yticklabel pos=right,ymin=.1]
			\pgfmathsetmacro{\aspect}{8}
			\addplot [mark=o,Dark2-C] table[x=k,y=w43,col sep=comma] {data-10-\aspect.dat};
			\addplot [mark=o,Dark2-D] table[x=k,y=w32,col sep=comma] {data-10-\aspect.dat};
			
			\addplot [mark=+,Dark2-A,style=thick,mark options={thick}] table[x=k,y=cheb4,col sep=comma] {data-10-\aspect.dat};
			\addplot [mark=triangle*,Dark2-B,style=thick] table[x=k,y=opt,col sep=comma] {data-10-\aspect.dat};
			\addplot [mark=x,Dark2-C] table[x=k,y=cheb1_3,col sep=comma] {data-10-\aspect.dat};
			\addplot [mark=x,Dark2-D] table[x=k,y=cheb1_10,col sep=comma] {data-10-\aspect.dat};
			\addplot [mark=x,Dark2-E] table[x=k,y=cheb1_30,col sep=comma] {data-10-\aspect.dat};
			\addplot [mark=x,Dark2-F] table[x=k,y=cheb1_100,col sep=comma] {data-10-\aspect.dat};
			\addplot [mark=x,Dark2-G] table[x=k,y=cheb1_300,col sep=comma] {data-10-\aspect.dat};
			
			\addplot [mark=diamond,Dark2-C] table[x=k,y=ufm_3,col sep=comma] {data-10-\aspect.dat};
			\addplot [mark=diamond,Dark2-D] table[x=k,y=ufm_10,col sep=comma] {data-10-\aspect.dat};
			\addplot [mark=diamond,Dark2-E] table[x=k,y=ufm_30,col sep=comma] {data-10-\aspect.dat};
			
		\end{groupplot}
		\coordinate (c3) at ($(c1)!.5!(c2)$);
		\node[below] at (c3 |- current bounding box.south)
		{\pgfplotslegendfromname{leg}};	
		
	\end{tikzpicture}
	
	\caption{V-Cycle error contraction factors for Poisson discretized on a uniform grid, comparison of different polynomials.}
	\label{fig:nr-simp2}
\end{figure}

In \cref{fig:nr-simp2} we compare the fourth-kind Chebyshev and optimized iterations with two other polynomial methods. The first, marked ``cheb1,'' is the standard Chebyshev iteration based on polynomials of the first kind, targeting the range of eigenvalues $[\kappa^{-1}, 1]$ for $\kappa=3,10,30,100,300$ (recall that we have scaled $B$ so that $\rho(BA)=1$). The second, marked ``ufm,'' is the method suggested by Kraus, Vassilevski, and Zikatanov \cite{kraus_polynomial_2012}, which takes $p_k(x) = 1 - q_{k-1}(x) x$ where $q_{k-1}(x)$ is the best degree $k-1$ polynomial approximation to $x^{-1}$ in the uniform norm on the interval $[\kappa^{-1}, 1]$. Equivalently, $p_k$ for this method is characterized as minimizing \[\sup_{\kappa^{-1} \le \lambda \le 1} \lambda^{-1} |p_k(\lambda)| . \]
Recall that the (scaled and shifted) Chebyshev polynomials of the fourth kind are characterized as minimizing a similar supremum, but over the whole interval $[0, 1]$ and with weight $\lambda^{1/2}$ instead of $\lambda^{-1}$. The Chebyshev polynomials of the first kind minimize the same supremum but with weight $1$. The weight $\lambda^{-1}$ is somewhat counter-intuitive, in that it places even more emphasis on damping the eigenmodes associated with small eigenvalues than the Chebyshev polynomials of the first kind. It also causes the polynomial iteration by itself to be divergent unless the degree is high enough,
\[ k \ge 2 \log \frac{\sqrt{2}}{\sqrt{\kappa}+1} \bigg/ \log \frac{\sqrt{\kappa}-1}{\sqrt{\kappa}+1} . \]
For the considered parameters $\kappa=3,10,30$, we must have $k \ge 2, 4, 9$ respectively for these polynomials to give convergent iterations.

\Cref{fig:nr-simp2} illustrates that the optimal parameter $\kappa$ for both the ``cheb1'' and ``ufm'' cases depends both on the specific problem as well as the number of smoothing steps $k$. The same parameter $\kappa$ can be optimal in one circumstance and very far from optimal in another. Compare, for example, the performance of $\kappa=100$ in the case of aspect ratio $8$ to its performance when the aspect ratio is $1$. On the other hand, the parameter-free fourth kind Chebyshev and optimized polynomials, while not always optimal, are competitive or at least not very suboptimal in every case. For the ``ufm'' case, we clearly see the effect of a minimum degree required for convergence. For fixed $\kappa$, as $k$ is increased the ``ufm'' polynomials eventually perform slightly better than the ``cheb1'' polynomials.

\begin{figure}[ptb]
	\centering
	\begin{tikzpicture}
		\begin{groupplot}[%
			group style={%
				group size=2 by 2,
				vertical sep=2cm
			},
			width=.5\linewidth
			]
			\nextgroupplot[title={$\tfrac{\Delta y}{\Delta x} = 1, C\approx 2$},ymode=log,ylabel={$1-C_N^{-1}$},
			legend style={font=\tiny,at={($(0,0)+(1cm,1cm)$)}, legend columns=4,fill=none,draw=black,anchor=center,align=center},
			legend to name=leg2
			]
			\pgfmathsetmacro{\aspect}{1}
			
			\addlegendimage{mark=x,Dark2-C}; \addlegendentry{cheb1, $\kappa=3$};
			\addlegendimage{mark=x,Dark2-D}; \addlegendentry{cheb1, $\kappa=10$};
			\addlegendimage{mark=x,Dark2-E}; \addlegendentry{cheb1, $\kappa=30$};
			\addlegendimage{mark=x,Dark2-F}; \addlegendentry{cheb1, $\kappa=100$};
			
			\addlegendimage{mark=+,Dark2-A,style=thick,mark options={thick}}; \addlegendentry{cheb4};
			\addlegendimage{mark=triangle*,Dark2-B,style=thick}; \addlegendentry{opt};

			\addplot [ domain=1:10, samples=201, Dark2-A, style=thick] {xdiv1px((\aspect^2)/((4/3)*x*(x+1))-xdiv1px(-pcheb(x, acos(1-2/(\aspect^2)))^2))};
			\addplot [mark=+,only marks, Dark2-A, mark options={thick}] table[x=k,y expr=1-1/\thisrow{cheb4},col sep=comma] {data-10-CN-\aspect.dat};
			
			\addplot [mark=triangle*,only marks, Dark2-B] table[x=k,y expr=1-1/\thisrow{opt},col sep=comma] {data-10-CN-\aspect.dat};
			\addplot [Dark2-B, style=thick] table[x=k,y expr=1-1/\thisrow{a\aspect},col sep=comma] {data-theory-CN-opt.dat};
			
			\addplot [ domain=1:10, samples=201,Dark2-C,dashed] {\aspect^2/(\aspect^2 + 2*x*sqrt(3))};
			\addplot [ domain=1:10, samples=201,Dark2-D,dashed] {\aspect^2/(\aspect^2 + 2*x*sqrt(10))};
			\addplot [ domain=1:10, samples=201,Dark2-E,dashed] {\aspect^2/(\aspect^2 + 2*x*sqrt(30))};
			\addplot [ domain=1:10, samples=201,Dark2-F,dashed] {\aspect^2/(\aspect^2 + 2*x*sqrt(100))};
			
			\addplot [ domain=1:10, samples=201,Dark2-C] {xdiv1px(cheb1cnm1(x,3,\aspect))};
			\addplot [ domain=1:10, samples=201,Dark2-D] {xdiv1px(cheb1cnm1(x,10,\aspect))};
			\addplot [ domain=1:10, samples=201,Dark2-E] {xdiv1px(cheb1cnm1(x,30,\aspect))};
			\addplot [ domain=1:10, samples=201,Dark2-F] {xdiv1px(cheb1cnm1(x,100,\aspect))};
			
			\addplot [mark=x, only marks, Dark2-C] table[x=k,y expr=1-1/\thisrow{cheb1_3},col sep=comma] {data-10-CN-\aspect.dat};
			\addplot [mark=x, only marks, Dark2-D] table[x=k,y expr=1-1/\thisrow{cheb1_10},col sep=comma] {data-10-CN-\aspect.dat};
			\addplot [mark=x, only marks, Dark2-E] table[x=k,y expr=1-1/\thisrow{cheb1_30},col sep=comma] {data-10-CN-\aspect.dat};
			\addplot [mark=x, only marks, Dark2-F] table[x=k,y expr=1-1/\thisrow{cheb1_100},col sep=comma] {data-10-CN-\aspect.dat};

			\coordinate (c1) at (rel axis cs:0,1);
			
			\nextgroupplot[title={$\tfrac{\Delta y}{\Delta x} = 2, C\approx 8$}, ymode=log,yticklabel pos=right]
			\pgfmathsetmacro{\aspect}{2}
			\addplot [ domain=1:10, samples=201, Dark2-A, style=thick] {xdiv1px((\aspect^2)/((4/3)*x*(x+1))-xdiv1px(-pcheb(x, acos(1-2/(\aspect^2)))^2))};
			\addplot [mark=+,only marks, Dark2-A, mark options={thick}] table[x=k,y expr=1-1/\thisrow{cheb4},col sep=comma] {data-10-CN-\aspect.dat};
			
			\addplot [mark=triangle*,only marks, Dark2-B] table[x=k,y expr=1-1/\thisrow{opt},col sep=comma] {data-10-CN-\aspect.dat};
			\addplot [Dark2-B, style=thick] table[x=k,y expr=1-1/\thisrow{a\aspect},col sep=comma] {data-theory-CN-opt.dat};
			
			\addplot [ domain=1:10, samples=201,Dark2-C,dashed] {\aspect^2/(\aspect^2 + 2*x*sqrt(3))};
			\addplot [ domain=1:10, samples=201,Dark2-D,dashed] {\aspect^2/(\aspect^2 + 2*x*sqrt(10))};
			\addplot [ domain=1:10, samples=201,Dark2-E,dashed] {\aspect^2/(\aspect^2 + 2*x*sqrt(30))};
			\addplot [ domain=1:10, samples=201,Dark2-F,dashed] {\aspect^2/(\aspect^2 + 2*x*sqrt(100))};
			
			\addplot [ domain=1:10, samples=201,Dark2-C] {xdiv1px(cheb1cnm1h(x,3,\aspect))};
			\addplot [ domain=1:10, samples=201,Dark2-D] {xdiv1px(cheb1cnm1(x,10,\aspect))};
			\addplot [ domain=1:10, samples=201,Dark2-E] {xdiv1px(cheb1cnm1(x,30,\aspect))};
			\addplot [ domain=1:10, samples=201,Dark2-F] {xdiv1px(cheb1cnm1(x,100,\aspect))};
			
			\addplot [mark=x, only marks, Dark2-C] table[x=k,y expr=1-1/\thisrow{cheb1_3},col sep=comma] {data-10-CN-\aspect.dat};
			\addplot [mark=x, only marks, Dark2-D] table[x=k,y expr=1-1/\thisrow{cheb1_10},col sep=comma] {data-10-CN-\aspect.dat};
			\addplot [mark=x, only marks, Dark2-E] table[x=k,y expr=1-1/\thisrow{cheb1_30},col sep=comma] {data-10-CN-\aspect.dat};
			\addplot [mark=x, only marks, Dark2-F] table[x=k,y expr=1-1/\thisrow{cheb1_100},col sep=comma] {data-10-CN-\aspect.dat};
			
			\coordinate (c2) at (rel axis cs:1,1);
			
			\nextgroupplot[title={$\tfrac{\Delta y}{\Delta x} = 4, C\approx 32$}, ymode=log,ylabel={$1-C_N^{-1}$},xlabel=$k$]
			\pgfmathsetmacro{\aspect}{4}
			\addplot [ domain=1:10, samples=201, Dark2-A, style=thick] {xdiv1px((\aspect^2)/((4/3)*x*(x+1))-xdiv1px(-pcheb(x, acos(1-2/(\aspect^2)))^2))};
			\addplot [mark=+,only marks, Dark2-A, mark options={thick}] table[x=k,y expr=1-1/\thisrow{cheb4},col sep=comma] {data-10-CN-\aspect.dat};
			
			\addplot [mark=triangle*,only marks, Dark2-B] table[x=k,y expr=1-1/\thisrow{opt},col sep=comma] {data-10-CN-\aspect.dat};
			\addplot [Dark2-B, style=thick] table[x=k,y expr=1-1/\thisrow{a\aspect},col sep=comma] {data-theory-CN-opt.dat};
			
			\addplot [ domain=1:10, samples=201,Dark2-C,dashed] {\aspect^2/(\aspect^2 + 2*x*sqrt(3))};
			\addplot [ domain=1:10, samples=201,Dark2-D,dashed] {\aspect^2/(\aspect^2 + 2*x*sqrt(10))};
			\addplot [ domain=1:10, samples=201,Dark2-E,dashed] {\aspect^2/(\aspect^2 + 2*x*sqrt(30))};
			\addplot [ domain=1:10, samples=201,Dark2-F,dashed] {\aspect^2/(\aspect^2 + 2*x*sqrt(100))};
			
			\addplot [ domain=1:10, samples=201,Dark2-C] {xdiv1px(cheb1cnm1h(x,3,\aspect))};
			\addplot [ domain=1:10, samples=201,Dark2-D] {xdiv1px(cheb1cnm1h(x,10,\aspect))};
			\addplot [ domain=1:10, samples=201,Dark2-E] {xdiv1px(cheb1cnm1(x,30,\aspect))};
			\addplot [ domain=1:10, samples=201,Dark2-F] {xdiv1px(cheb1cnm1(x,100,\aspect))};
			
			\addplot [mark=x, only marks, Dark2-C] table[x=k,y expr=1-1/\thisrow{cheb1_3},col sep=comma] {data-10-CN-\aspect.dat};
			\addplot [mark=x, only marks, Dark2-D] table[x=k,y expr=1-1/\thisrow{cheb1_10},col sep=comma] {data-10-CN-\aspect.dat};
			\addplot [mark=x, only marks, Dark2-E] table[x=k,y expr=1-1/\thisrow{cheb1_30},col sep=comma] {data-10-CN-\aspect.dat};
			\addplot [mark=x, only marks, Dark2-F] table[x=k,y expr=1-1/\thisrow{cheb1_100},col sep=comma] {data-10-CN-\aspect.dat};
			
			\nextgroupplot[title={$\tfrac{\Delta y}{\Delta x} = 8, C\approx 128$}, ymode=log,xlabel=$k$,yticklabel pos=right,ymin=.1]
			\pgfmathsetmacro{\aspect}{8}
			\addplot [ domain=1:10, samples=201, Dark2-A, style=thick] {xdiv1px((\aspect^2)/((4/3)*x*(x+1))-xdiv1px(-pcheb(x, acos(1-2/(\aspect^2)))^2))};
			\addplot [mark=+,only marks, Dark2-A, mark options={thick}] table[x=k,y expr=1-1/\thisrow{cheb4},col sep=comma] {data-10-CN-\aspect.dat};
			
			\addplot [mark=triangle*,only marks, Dark2-B] table[x=k,y expr=1-1/\thisrow{opt},col sep=comma] {data-10-CN-\aspect.dat};
			\addplot [Dark2-B, style=thick] table[x=k,y expr=1-1/\thisrow{a\aspect},col sep=comma] {data-theory-CN-opt.dat};
			
			\addplot [ domain=1:10, samples=201,Dark2-C,dashed] {\aspect^2/(\aspect^2 + 2*x*sqrt(3))};
			\addplot [ domain=1:10, samples=201,Dark2-D,dashed] {\aspect^2/(\aspect^2 + 2*x*sqrt(10))};
			\addplot [ domain=1:10, samples=201,Dark2-E,dashed] {\aspect^2/(\aspect^2 + 2*x*sqrt(30))};
			\addplot [ domain=1:10, samples=201,Dark2-F,dashed] {\aspect^2/(\aspect^2 + 2*x*sqrt(100))};
			
			\addplot [ domain=1:10, samples=201,Dark2-C] {xdiv1px(cheb1cnm1h(x,3,\aspect))};
			\addplot [ domain=1:10, samples=201,Dark2-D] {xdiv1px(cheb1cnm1h(x,10,\aspect))};
			\addplot [ domain=1:10, samples=201,Dark2-E] {xdiv1px(cheb1cnm1h(x,30,\aspect))};
			\addplot [ domain=1:10, samples=201,Dark2-F] {xdiv1px(cheb1cnm1(x,100,\aspect))};
			
			\addplot [mark=x, only marks, Dark2-C] table[x=k,y expr=1-1/\thisrow{cheb1_3},col sep=comma] {data-10-CN-\aspect.dat};
			\addplot [mark=x, only marks, Dark2-D] table[x=k,y expr=1-1/\thisrow{cheb1_10},col sep=comma] {data-10-CN-\aspect.dat};
			\addplot [mark=x, only marks, Dark2-E] table[x=k,y expr=1-1/\thisrow{cheb1_30},col sep=comma] {data-10-CN-\aspect.dat};
			\addplot [mark=x, only marks, Dark2-F] table[x=k,y expr=1-1/\thisrow{cheb1_100},col sep=comma] {data-10-CN-\aspect.dat};
			
		\end{groupplot}
		\coordinate (c3) at ($(c1)!.5!(c2)$);
		\node[below] at (c3 |- current bounding box.south)
		{\pgfplotslegendfromname{leg2}};	
		
	\end{tikzpicture}
	
	\caption{McCormick V-cycle bounds for Poisson discretized on a uniform grid.}
	\label{fig:nr-simp3}
\end{figure}
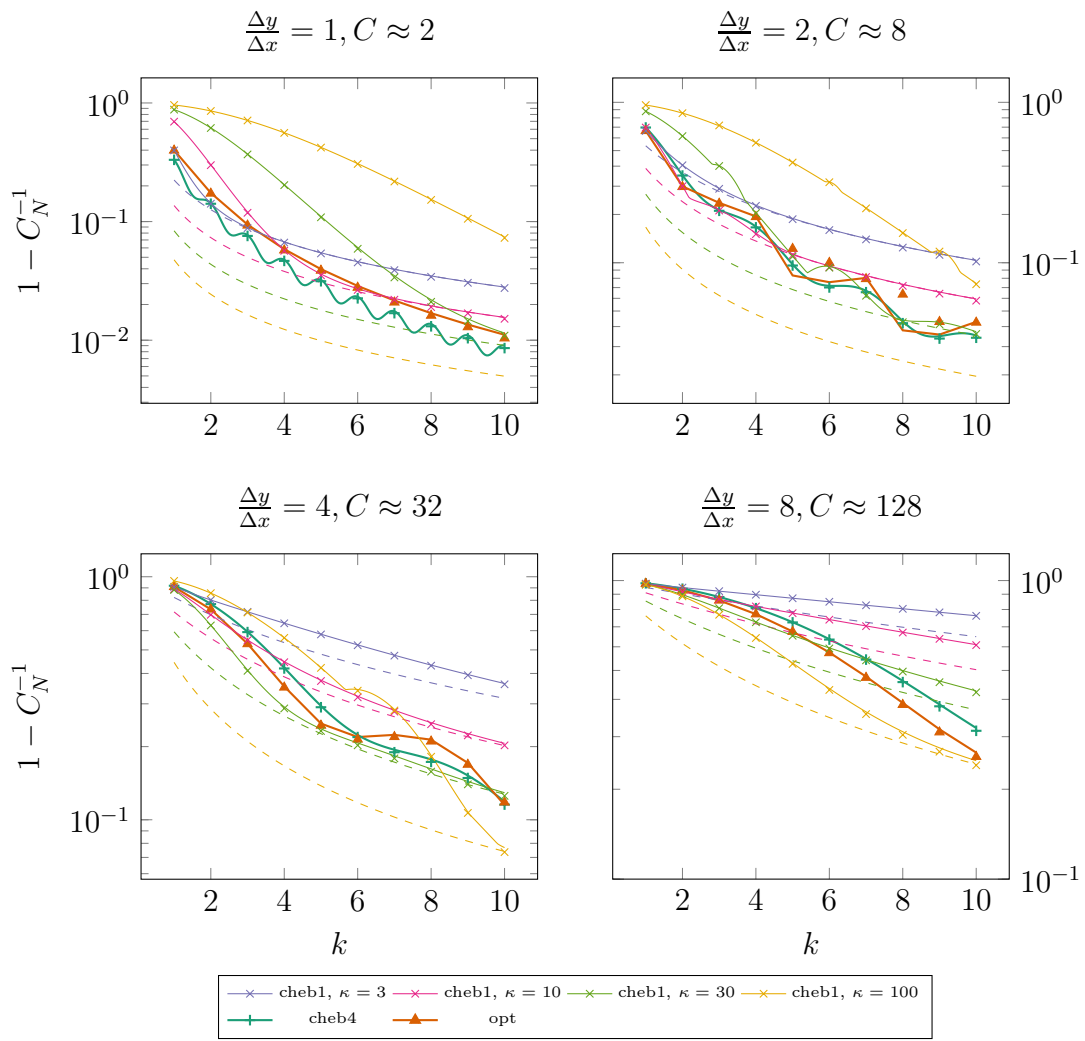

In \cref{fig:nr-simp3} we consider the McCormick bound $\norm{E}_A^2 \le 1 - C_N^{-1}$ of \cref{lem:McCormick-bound} for a subset of the polynomial smoothers considered in \cref{fig:nr-simp2}, in order to gain more insight into their behavior. The markers correspond to measured values of $C_N$, determined using a generalized Lanczos iteration with $N^{-1}$ implemented by transforming to the eigenbasis. Many of the general trends in \cref{fig:nr-simp2} are reflected in the upper bounds plotted in \cref{fig:nr-simp3}, though there are some discrepancies resulting from the variance in the sharpness of the bound.

At first glance the curves in \cref{fig:nr-simp3} might appear to connect the markers. They are actually plots of $1-\lambda^{-1}$ where $\lambda$ is the maximum of
\begin{equation} \label{eq:CN-fourier-analysis-eigenvalues} \begin{aligned} \lambda_1 &= 1 - \frac{\alpha^2}{2 p_k'(0)} + \frac{p_k(\alpha^{-2})^2}{1-p_k(\alpha^{-2})^2}, \text{ and} \\
		\lambda_2 &=
	1 - \frac{1}{2 p_k'(0)} + \frac{p_k(1)^2}{1-p_k(1)^2}, \end{aligned} \end{equation}
where $\alpha = \Delta y / \Delta x \ge 1$ is the aspect ratio. These are two asymptotic eigenvalues of $N^{-1} \pi_f A^{-1}$ associated with the smallest horizontal and vertical wavenumbers in the limit as the number of grid points goes to infinity, determined by a Fourier analysis. Note that $C_N$ is the spectral radius of $N^{-1} \pi_f A^{-1}$. There's no assurance that the global maximum eigenvalue of $N^{-1} \pi_f A^{-1}$ is associated with the smallest wavenumbers, but the close correspondence of the curves in \cref{fig:nr-simp3} to the markers indicates this to likely be the case in most of the instances considered. The notable exceptions are the ``opt'' polynomials at aspect ratio $\alpha=2$ and degrees $k=5,6,8,9$. Indeed a close examination of these cases reveals the global maximum eigenvalues to occur at wavenumbers far from the smallest.

To the extent that the McCormick bound reflects the actual multigrid convergence, we can explain the convergence behavior by examining $\lambda_1$ and $\lambda_2$, at least in those cases where their maximum well approximates $C_N$. First note that the terms involving $p_k'(0)$ and $p_k(1)$ behave monotonically with respect to $k$, for all the polynomial classes considered. E.g., with any parameter being fixed, $p_k(1)$ always decreases as $k$ increases. On the other hand, the term involving $p_k(\alpha^{-2})$ can exhibit oscillations with $k$. Much of the interesting and perhaps unexpected behavior can be attributed to this term. Note that we have chosen to interpolate the plots of $1-\lambda^{-1}$ to non-integer $k$ using analytic expressions where available, in order to highlight the oscillatory behavior. For example, for the Chebyshev polynomials of the fourth kind, $p_k(\alpha^{-2})$ is proportional to $\sin[(k+\tfrac{1}{2}) \arccos (1-2\alpha^{-2})]$ by \eqref{eq:W-defn}. As a result, the wavelength of the oscillations increases with the aspect ratio $\alpha$. As another example, for the ``cheb1'' case with $\kappa=100$ and aspect ratio $\alpha=2$, $\lambda$ is mostly equal to $\lambda_2$ except within two narrow intervals near $k=6$ and $k=9$ where the peaks of the oscillations of $\lambda_1$ stick out, explaining the subtle kinks in the plot of $1-C_N^{-1}$ at these locations.

For the polynomial smoothers based on the Chebyshev polynomials of the first kind, \cref{fig:nr-simp2} illustrates the convergence often improves rapidly with $k$ at first before settling into a slower asymptotic regime. This is reflected in the behavior of the McCormick upper bound in \cref{fig:nr-simp3}, which can be explained as follows. For the Chebyshev polynomials of the first kind,
\[ p_k'(0) = k \sqrt{\kappa} \tanh \left[k \arccosh \frac{\kappa+1}{\kappa-1}\right] . \]
For fixed $\kappa$, as $k$ increases, in particular when it becomes large compared to $1/\arccosh \frac{\kappa+1}{\kappa-1} \sim \sqrt{\kappa}/2$, $\tanh$ approaches $1$ exponentially fast so that $p_k'(0) \sim k\sqrt{\kappa}$ plus exponentially small terms (e.s.t.). Meanwhile $p_k(\alpha^{-2})$ and $p_k(1)$ both eventually become exponentially small as $k$ increases, so that $\lambda \sim 1 + \frac{\alpha^2}{2 k \sqrt{\kappa}}$ plus e.s.t. Hence, since $C \approx 2 \alpha^2$, for fixed $\kappa$ we expect the McCormick bound $1-C_N^{-1}$ for the Chebyshev polynomials of the first kind to eventually behave approximately like
\[ \frac{C}{C + 4 k \sqrt{\kappa}} \]
as $k$ increases. The dashed curves in \cref{fig:nr-simp3} are plots of this expression, with colors matching the corresponding plots of $\lambda^{-1}$. Many of the ``cheb1'' curves can be seen to fall into this asymptotic behavior; the others can be expected to do so at higher values of $k$. The asymptotic behavior explains two general features of the McCormick bound reflecting features of the multigrid convergence rates. For a given aspect ratio, the first-kind Chebyshev method with a higher $\kappa$ always eventually has a better McCormick bound than a lower $\kappa$, at high enough $k$, though the cut-off may be at impractically large $k$. Secondly, the fourth-kind Chebyshev and optimized polynomials always eventually have lower McCormick bounds than any given first-kind Chebyshev method with fixed $\kappa$, again for large enough $k$, since \cref{cor:cheb-bound} implies an upper bound for $1-C_N^{-1}$ for the fourth-kind and optimized polynomials that is $O(k^{-2})$.

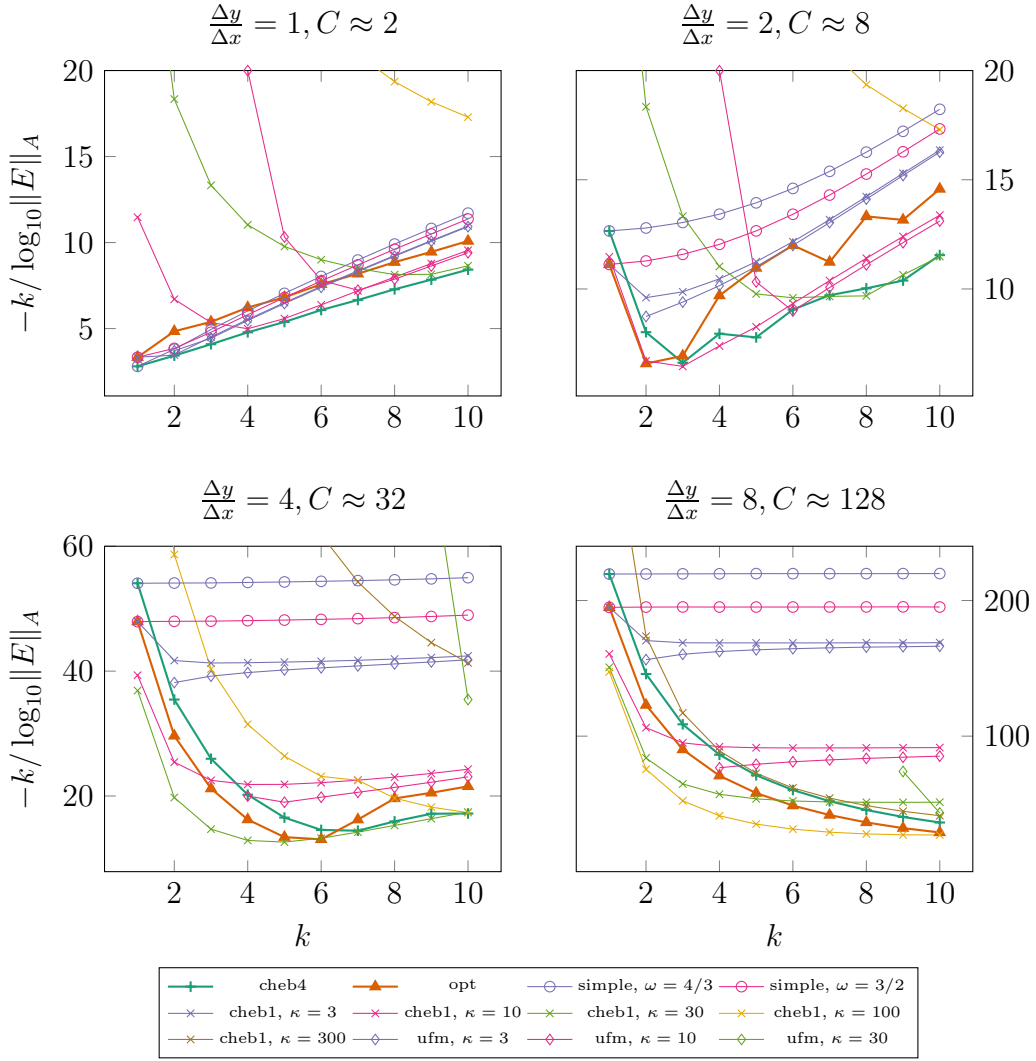
\begin{figure}[ptb]
	\centering
	\begin{tikzpicture}
		\begin{groupplot}[%
			group style={%
				group size=2 by 2,
				vertical sep=2cm
			},
			width=.5\linewidth
			]
			\nextgroupplot[title={$\tfrac{\Delta y}{\Delta x} = 1, C\approx 2$},ymax=20,ylabel={$-k/\log_{10} \norm{E}_A$},
			legend style={font=\tiny,at={($(0,0)+(1cm,1cm)$)}, legend columns=4,fill=none,draw=black,anchor=center,align=center},
			legend to name=leg3
			]
			\pgfmathsetmacro{\aspect}{1}
			
			\addlegendimage{mark=+,Dark2-A,style=thick,mark options={thick}}; \addlegendentry{cheb4};
			\addlegendimage{mark=triangle*,Dark2-B,style=thick}; \addlegendentry{opt};
			
			\addlegendimage{mark=o,Dark2-C}; \addlegendentry{simple, $\omega=4/3$};
			\addlegendimage{mark=o,Dark2-D}; \addlegendentry{simple, $\omega=3/2$};
			
			\addlegendimage{mark=x,Dark2-C}; \addlegendentry{cheb1, $\kappa=3$};
			\addlegendimage{mark=x,Dark2-D}; \addlegendentry{cheb1, $\kappa=10$};
			\addlegendimage{mark=x,Dark2-E}; \addlegendentry{cheb1, $\kappa=30$};
			\addlegendimage{mark=x,Dark2-F}; \addlegendentry{cheb1, $\kappa=100$};
			\addlegendimage{mark=x,Dark2-G}; \addlegendentry{cheb1, $\kappa=300$};
			
			\addlegendimage{mark=diamond,Dark2-C}; \addlegendentry{ufm, $\kappa=3$};
			\addlegendimage{mark=diamond,Dark2-D}; \addlegendentry{ufm, $\kappa=10$};
			\addlegendimage{mark=diamond,Dark2-E}; \addlegendentry{ufm, $\kappa=30$};
			
			\addplot [mark=+,Dark2-A,style=thick,mark options={thick}] table[x=k,y expr=-2*\thisrow{k}/log10(\thisrow{cheb4}),col sep=comma] {data-10-\aspect.dat};
			\addplot [mark=triangle*,Dark2-B,style=thick] table[x=k,y expr=-2*\thisrow{k}/log10(\thisrow{opt}),col sep=comma] {data-10-\aspect.dat};
			\addplot [mark=x,Dark2-C] table[x=k,y expr=-2*\thisrow{k}/log10(\thisrow{cheb1_3}),col sep=comma] {data-10-\aspect.dat};
			\addplot [mark=x,Dark2-D] table[x=k,y expr=-2*\thisrow{k}/log10(\thisrow{cheb1_10}),col sep=comma] {data-10-\aspect.dat};
			\addplot [mark=x,Dark2-E] table[x=k,y expr=-2*\thisrow{k}/log10(\thisrow{cheb1_30}),col sep=comma] {data-10-\aspect.dat};
			\addplot [mark=x,Dark2-F] table[x=k,y expr=-2*\thisrow{k}/log10(\thisrow{cheb1_100}),col sep=comma] {data-10-\aspect.dat};
			\addplot [mark=x,Dark2-G] table[x=k,y expr=-2*\thisrow{k}/log10(\thisrow{cheb1_300}),col sep=comma] {data-10-\aspect.dat};
			\addplot [mark=o,Dark2-C] table[x=k,y expr=-2*\thisrow{k}/log10(\thisrow{w43}),col sep=comma] {data-10-\aspect.dat};
			\addplot [mark=o,Dark2-D] table[x=k,y expr=-2*\thisrow{k}/log10(\thisrow{w32}),col sep=comma] {data-10-\aspect.dat};
			\addplot [mark=diamond,Dark2-C] table[x=k,y expr=-2*\thisrow{k}/log10(\thisrow{ufm_3}),col sep=comma,,skip coords between index={0}{1}] {data-10-\aspect.dat};
			\addplot [mark=diamond,Dark2-D] table[x=k,y expr=-2*\thisrow{k}/log10(\thisrow{ufm_10}),col sep=comma,skip coords between index={0}{3}] {data-10-\aspect.dat};
			\addplot [mark=diamond,Dark2-E] table[x=k,y expr=-2*\thisrow{k}/log10(\thisrow{ufm_30}),col sep=comma,,skip coords between index={0}{8}] {data-10-\aspect.dat};
			
			\coordinate (c1) at (rel axis cs:0,1);
			
			\nextgroupplot[title={$\tfrac{\Delta y}{\Delta x} = 2, C\approx 8$}, ymax=20,yticklabel pos=right]
			\pgfmathsetmacro{\aspect}{2}
			\addplot [mark=+,Dark2-A,style=thick,mark options={thick}] table[x=k,y expr=-2*\thisrow{k}/log10(\thisrow{cheb4}),col sep=comma] {data-10-\aspect.dat};
			\addplot [mark=triangle*,Dark2-B,style=thick] table[x=k,y expr=-2*\thisrow{k}/log10(\thisrow{opt}),col sep=comma] {data-10-\aspect.dat};
			\addplot [mark=x,Dark2-C] table[x=k,y expr=-2*\thisrow{k}/log10(\thisrow{cheb1_3}),col sep=comma] {data-10-\aspect.dat};
			\addplot [mark=x,Dark2-D] table[x=k,y expr=-2*\thisrow{k}/log10(\thisrow{cheb1_10}),col sep=comma] {data-10-\aspect.dat};
			\addplot [mark=x,Dark2-E] table[x=k,y expr=-2*\thisrow{k}/log10(\thisrow{cheb1_30}),col sep=comma] {data-10-\aspect.dat};
			\addplot [mark=x,Dark2-F] table[x=k,y expr=-2*\thisrow{k}/log10(\thisrow{cheb1_100}),col sep=comma] {data-10-\aspect.dat};
			\addplot [mark=x,Dark2-G] table[x=k,y expr=-2*\thisrow{k}/log10(\thisrow{cheb1_300}),col sep=comma] {data-10-\aspect.dat};
			\addplot [mark=o,Dark2-C] table[x=k,y expr=-2*\thisrow{k}/log10(\thisrow{w43}),col sep=comma] {data-10-\aspect.dat};
			\addplot [mark=o,Dark2-D] table[x=k,y expr=-2*\thisrow{k}/log10(\thisrow{w32}),col sep=comma] {data-10-\aspect.dat};
			\addplot [mark=diamond,Dark2-C] table[x=k,y expr=-2*\thisrow{k}/log10(\thisrow{ufm_3}),col sep=comma,skip coords between index={0}{1}] {data-10-\aspect.dat};
			\addplot [mark=diamond,Dark2-D] table[x=k,y expr=-2*\thisrow{k}/log10(\thisrow{ufm_10}),col sep=comma,skip coords between index={0}{3}] {data-10-\aspect.dat};
			\addplot [mark=diamond,Dark2-E] table[x=k,y expr=-2*\thisrow{k}/log10(\thisrow{ufm_30}),col sep=comma,,skip coords between index={0}{8}] {data-10-\aspect.dat};
			
			\coordinate (c2) at (rel axis cs:1,1);
			
			\nextgroupplot[title={$\tfrac{\Delta y}{\Delta x} = 4, C\approx 32$}, ymax=60,ylabel={$-k/\log_{10} \norm{E}_A$},xlabel=$k$]
			\pgfmathsetmacro{\aspect}{4}
			\addplot [mark=+,Dark2-A,style=thick,mark options={thick}] table[x=k,y expr=-2*\thisrow{k}/log10(\thisrow{cheb4}),col sep=comma] {data-10-\aspect.dat};
			\addplot [mark=triangle*,Dark2-B,style=thick] table[x=k,y expr=-2*\thisrow{k}/log10(\thisrow{opt}),col sep=comma] {data-10-\aspect.dat};
			\addplot [mark=x,Dark2-C] table[x=k,y expr=-2*\thisrow{k}/log10(\thisrow{cheb1_3}),col sep=comma] {data-10-\aspect.dat};
			\addplot [mark=x,Dark2-D] table[x=k,y expr=-2*\thisrow{k}/log10(\thisrow{cheb1_10}),col sep=comma] {data-10-\aspect.dat};
			\addplot [mark=x,Dark2-E] table[x=k,y expr=-2*\thisrow{k}/log10(\thisrow{cheb1_30}),col sep=comma] {data-10-\aspect.dat};
			\addplot [mark=x,Dark2-F] table[x=k,y expr=-2*\thisrow{k}/log10(\thisrow{cheb1_100}),col sep=comma] {data-10-\aspect.dat};
			\addplot [mark=x,Dark2-G] table[x=k,y expr=-2*\thisrow{k}/log10(\thisrow{cheb1_300}),col sep=comma] {data-10-\aspect.dat};
			\addplot [mark=o,Dark2-C] table[x=k,y expr=-2*\thisrow{k}/log10(\thisrow{w43}),col sep=comma] {data-10-\aspect.dat};
			\addplot [mark=o,Dark2-D] table[x=k,y expr=-2*\thisrow{k}/log10(\thisrow{w32}),col sep=comma] {data-10-\aspect.dat};
			\addplot [mark=diamond,Dark2-C] table[x=k,y expr=-2*\thisrow{k}/log10(\thisrow{ufm_3}),col sep=comma,skip coords between index={0}{1}] {data-10-\aspect.dat};
			\addplot [mark=diamond,Dark2-D] table[x=k,y expr=-2*\thisrow{k}/log10(\thisrow{ufm_10}),col sep=comma,skip coords between index={0}{3}] {data-10-\aspect.dat};
			\addplot [mark=diamond,Dark2-E] table[x=k,y expr=-2*\thisrow{k}/log10(\thisrow{ufm_30}),col sep=comma,,skip coords between index={0}{8}] {data-10-\aspect.dat};
			
			\nextgroupplot[title={$\tfrac{\Delta y}{\Delta x} = 8, C\approx 128$}, ymax=240,xlabel=$k$,yticklabel pos=right,ymin=.1]
			\pgfmathsetmacro{\aspect}{8}
			\addplot [mark=+,Dark2-A,style=thick,mark options={thick}] table[x=k,y expr=-2*\thisrow{k}/log10(\thisrow{cheb4}),col sep=comma] {data-10-\aspect.dat};
			\addplot [mark=triangle*,Dark2-B,style=thick] table[x=k,y expr=-2*\thisrow{k}/log10(\thisrow{opt}),col sep=comma] {data-10-\aspect.dat};
			\addplot [mark=x,Dark2-C] table[x=k,y expr=-2*\thisrow{k}/log10(\thisrow{cheb1_3}),col sep=comma] {data-10-\aspect.dat};
			\addplot [mark=x,Dark2-D] table[x=k,y expr=-2*\thisrow{k}/log10(\thisrow{cheb1_10}),col sep=comma] {data-10-\aspect.dat};
			\addplot [mark=x,Dark2-E] table[x=k,y expr=-2*\thisrow{k}/log10(\thisrow{cheb1_30}),col sep=comma] {data-10-\aspect.dat};
			\addplot [mark=x,Dark2-F] table[x=k,y expr=-2*\thisrow{k}/log10(\thisrow{cheb1_100}),col sep=comma] {data-10-\aspect.dat};
			\addplot [mark=x,Dark2-G] table[x=k,y expr=-2*\thisrow{k}/log10(\thisrow{cheb1_300}),col sep=comma] {data-10-\aspect.dat};
			\addplot [mark=o,Dark2-C] table[x=k,y expr=-2*\thisrow{k}/log10(\thisrow{w43}),col sep=comma] {data-10-\aspect.dat};
			\addplot [mark=o,Dark2-D] table[x=k,y expr=-2*\thisrow{k}/log10(\thisrow{w32}),col sep=comma] {data-10-\aspect.dat};
			\addplot [mark=diamond,Dark2-C] table[x=k,y expr=-2*\thisrow{k}/log10(\thisrow{ufm_3}),col sep=comma,skip coords between index={0}{1}] {data-10-\aspect.dat};
			\addplot [mark=diamond,Dark2-D] table[x=k,y expr=-2*\thisrow{k}/log10(\thisrow{ufm_10}),col sep=comma,skip coords between index={0}{3}] {data-10-\aspect.dat};
			\addplot [mark=diamond,Dark2-E] table[x=k,y expr=-2*\thisrow{k}/log10(\thisrow{ufm_30}),col sep=comma,,skip coords between index={0}{8}] {data-10-\aspect.dat};
			
		\end{groupplot}
		\coordinate (c3) at ($(c1)!.5!(c2)$);
		\node[below] at (c3 |- current bounding box.south)
		{\pgfplotslegendfromname{leg3}};	
		
	\end{tikzpicture}
	
	\caption{Smoother steps per decimal digit for Poisson discretized on a uniform grid.}
	\label{fig:nr-simp4}
\end{figure}

In \cref{fig:nr-simp4}, we consider the efficiency of the polynomial smoothers, taking into account the higher costs of using more smoother steps. Since the asymptotic convergence rate of the symmetric V-cycle is $\norm{E}_A^2$, asymptotically $-\log_{10} \norm{E}_A^2$ symmetric V-cycles are needed per decimal digit (reduction of the error by a factor of 10). Since each V-cycle requires $2k$ applications of the single-step smoother, $-k/\log_{10} \norm{E}_A$ thus gives the asymptotic number of smoother steps required per digit. We see in \cref{fig:nr-simp4} that the optimal number of smoothing steps $k$ increases as the aspect ratio and thereby the approximation constant $C$ increases. At aspect ratio 1, there's no benefit to using $k > 1$, at least in terms of total smoothing steps, whereas at aspect ratio 8, the optimum appears to be beyond even $k=10$. The simple smoothers show no benefit to $k>1$ with this metric in any case. We should note that considering only the total number of smoother steps ignores the costs of the prolongation and restriction operations, which scales with the number of V-cycles, hence further penalizing small $k$, and making the estimate of optimal $k$ here conservative. Interestingly, while the fourth-kind and optimized polynomials can be somewhat suboptimal at specific $k$, we note that their minimum values of $-k/\log_{10} \norm{E}_A$ are quite competitive for every aspect ratio.

\subsection{Coefficient jumps}

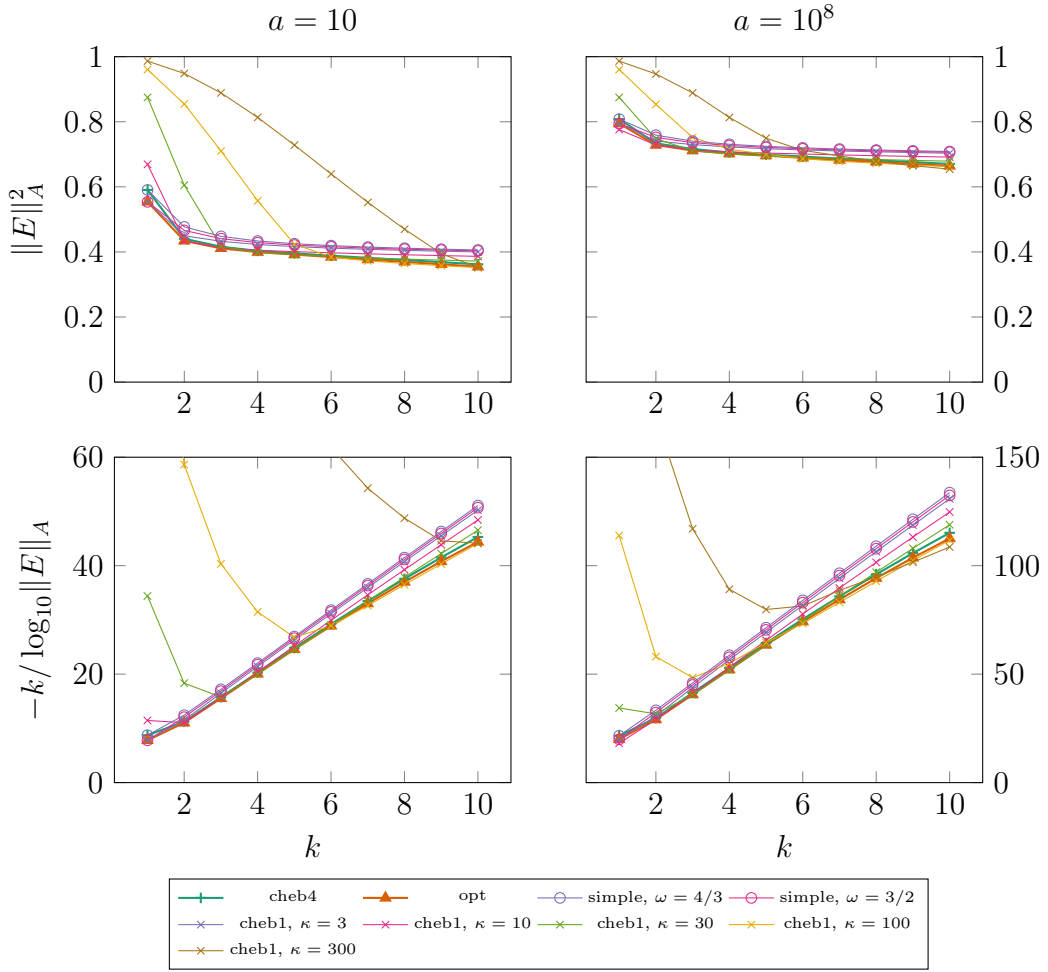
\begin{figure}[ptb]
	\centering
	\begin{tikzpicture}
		\begin{groupplot}[%
			group style={%
				group size=2 by 2,
			},
			width=.5\linewidth
			]
			\nextgroupplot[title={$a=10$},ymax=1,ymin=0,ylabel={$\norm{E}_A^2$},
			legend style={font=\tiny,at={($(0,0)+(1cm,1cm)$)}, legend columns=4,fill=none,draw=black,anchor=center,align=center},
			legend to name=leg-varcoef
			]
			
			\addlegendimage{mark=+,Dark2-A,style=thick,mark options={thick}}; \addlegendentry{cheb4};
			\addlegendimage{mark=triangle*,Dark2-B,style=thick}; \addlegendentry{opt};
			\addlegendimage{mark=o,Dark2-C}; \addlegendentry{simple, $\omega=4/3$};
			\addlegendimage{mark=o,Dark2-D}; \addlegendentry{simple, $\omega=3/2$};
			
			\addlegendimage{mark=x,Dark2-C}; \addlegendentry{cheb1, $\kappa=3$};
			\addlegendimage{mark=x,Dark2-D}; \addlegendentry{cheb1, $\kappa=10$};
			\addlegendimage{mark=x,Dark2-E}; \addlegendentry{cheb1, $\kappa=30$};
			\addlegendimage{mark=x,Dark2-F}; \addlegendentry{cheb1, $\kappa=100$};
			\addlegendimage{mark=x,Dark2-G}; \addlegendentry{cheb1, $\kappa=300$};

			\pgfmathsetmacro{\paramN}{32}
			\pgfmathsetmacro{\paramE}{32}
			\def\paramA{1e1}
			\def\paramSmoother{jacobi}
			
			\addplot [mark=+,Dark2-A,style=thick,mark options={thick}] table[x=k,y=cheb4,col sep=comma] {data-varcoef-\paramN-\paramE-\paramA-\paramSmoother.dat};
			\addplot [mark=triangle*,Dark2-B,style=thick] table[x=k,y=opt,col sep=comma] {data-varcoef-\paramN-\paramE-\paramA-\paramSmoother.dat};
			\addplot [mark=x,Dark2-C] table[x=k,y=cheb1_3,col sep=comma] {data-varcoef-\paramN-\paramE-\paramA-\paramSmoother.dat};
			\addplot [mark=x,Dark2-D] table[x=k,y=cheb1_10,col sep=comma] {data-varcoef-\paramN-\paramE-\paramA-\paramSmoother.dat};
			\addplot [mark=x,Dark2-E] table[x=k,y=cheb1_30,col sep=comma] {data-varcoef-\paramN-\paramE-\paramA-\paramSmoother.dat};
			\addplot [mark=x,Dark2-F] table[x=k,y=cheb1_100,col sep=comma] {data-varcoef-\paramN-\paramE-\paramA-\paramSmoother.dat};
			\addplot [mark=x,Dark2-G] table[x=k,y=cheb1_300,col sep=comma] {data-varcoef-\paramN-\paramE-\paramA-\paramSmoother.dat};
			
			\addplot [mark=o,Dark2-C] table[x=k,y=w43,col sep=comma] {data-varcoef-\paramN-\paramE-\paramA-\paramSmoother.dat};
			\addplot [mark=o,Dark2-D] table[x=k,y=w32,col sep=comma] {data-varcoef-\paramN-\paramE-\paramA-\paramSmoother.dat};
			
			\coordinate (c1) at (rel axis cs:0,1);
			
			\nextgroupplot[title={$a=10^8$},ymin=0,ymax=1,yticklabel pos=right]
			
			\pgfmathsetmacro{\paramN}{32}
			\pgfmathsetmacro{\paramE}{32}
			\def\paramA{1e8}
			\def\paramSmoother{jacobi}
			
			\addplot [mark=+,Dark2-A,style=thick,mark options={thick}] table[x=k,y=cheb4,col sep=comma] {data-varcoef-\paramN-\paramE-\paramA-\paramSmoother.dat};
			\addplot [mark=triangle*,Dark2-B,style=thick] table[x=k,y=opt,col sep=comma] {data-varcoef-\paramN-\paramE-\paramA-\paramSmoother.dat};
			\addplot [mark=x,Dark2-C] table[x=k,y=cheb1_3,col sep=comma] {data-varcoef-\paramN-\paramE-\paramA-\paramSmoother.dat};
			\addplot [mark=x,Dark2-D] table[x=k,y=cheb1_10,col sep=comma] {data-varcoef-\paramN-\paramE-\paramA-\paramSmoother.dat};
			\addplot [mark=x,Dark2-E] table[x=k,y=cheb1_30,col sep=comma] {data-varcoef-\paramN-\paramE-\paramA-\paramSmoother.dat};
			\addplot [mark=x,Dark2-F] table[x=k,y=cheb1_100,col sep=comma] {data-varcoef-\paramN-\paramE-\paramA-\paramSmoother.dat};
			\addplot [mark=x,Dark2-G] table[x=k,y=cheb1_300,col sep=comma] {data-varcoef-\paramN-\paramE-\paramA-\paramSmoother.dat};
			
			\addplot [mark=o,Dark2-C] table[x=k,y=w43,col sep=comma] {data-varcoef-\paramN-\paramE-\paramA-\paramSmoother.dat};
			\addplot [mark=o,Dark2-D] table[x=k,y=w32,col sep=comma] {data-varcoef-\paramN-\paramE-\paramA-\paramSmoother.dat};
			
			\coordinate (c2) at (rel axis cs:1,1);
			
			\nextgroupplot[ylabel={$-k/\log_{10} \norm{E}_A$},xlabel=$k$,ymax=60,ymin=0]
			
			\pgfmathsetmacro{\paramN}{32}
			\pgfmathsetmacro{\paramE}{32}
			\def\paramA{1e1}
			\def\paramSmoother{jacobi}
			
			\addplot [mark=+,Dark2-A,style=thick,mark options={thick}] table[x=k,y expr=-2*\thisrow{k}/log10(\thisrow{cheb4}),col sep=comma] {data-varcoef-\paramN-\paramE-\paramA-\paramSmoother.dat};
			\addplot [mark=triangle*,Dark2-B,style=thick] table[x=k,y expr=-2*\thisrow{k}/log10(\thisrow{opt}),col sep=comma] {data-varcoef-\paramN-\paramE-\paramA-\paramSmoother.dat};
			\addplot [mark=x,Dark2-C] table[x=k,y expr=-2*\thisrow{k}/log10(\thisrow{cheb1_3}),col sep=comma] {data-varcoef-\paramN-\paramE-\paramA-\paramSmoother.dat};
			\addplot [mark=x,Dark2-D] table[x=k,y expr=-2*\thisrow{k}/log10(\thisrow{cheb1_10}),col sep=comma] {data-varcoef-\paramN-\paramE-\paramA-\paramSmoother.dat};
			\addplot [mark=x,Dark2-E] table[x=k,y expr=-2*\thisrow{k}/log10(\thisrow{cheb1_30}),col sep=comma] {data-varcoef-\paramN-\paramE-\paramA-\paramSmoother.dat};
			\addplot [mark=x,Dark2-F] table[x=k,y expr=-2*\thisrow{k}/log10(\thisrow{cheb1_100}),col sep=comma] {data-varcoef-\paramN-\paramE-\paramA-\paramSmoother.dat};
			\addplot [mark=x,Dark2-G] table[x=k,y expr=-2*\thisrow{k}/log10(\thisrow{cheb1_300}),col sep=comma] {data-varcoef-\paramN-\paramE-\paramA-\paramSmoother.dat};
			
			\addplot [mark=o,Dark2-C] table[x=k,y expr=-2*\thisrow{k}/log10(\thisrow{w43}),col sep=comma] {data-varcoef-\paramN-\paramE-\paramA-\paramSmoother.dat};
			\addplot [mark=o,Dark2-D] table[x=k,y expr=-2*\thisrow{k}/log10(\thisrow{w32}),col sep=comma] {data-varcoef-\paramN-\paramE-\paramA-\paramSmoother.dat};
			
			\nextgroupplot[ymin=0,ymax=150,xlabel=$k$,yticklabel pos=right]
			
			\pgfmathsetmacro{\paramN}{32}
			\pgfmathsetmacro{\paramE}{32}
			\def\paramA{1e8}
			\def\paramSmoother{jacobi}
			
			\addplot [mark=+,Dark2-A,style=thick,mark options={thick}] table[x=k,y expr=-2*\thisrow{k}/log10(\thisrow{cheb4}),col sep=comma] {data-varcoef-\paramN-\paramE-\paramA-\paramSmoother.dat};
			\addplot [mark=triangle*,Dark2-B,style=thick] table[x=k,y expr=-2*\thisrow{k}/log10(\thisrow{opt}),col sep=comma] {data-varcoef-\paramN-\paramE-\paramA-\paramSmoother.dat};
			\addplot [mark=x,Dark2-C] table[x=k,y expr=-2*\thisrow{k}/log10(\thisrow{cheb1_3}),col sep=comma] {data-varcoef-\paramN-\paramE-\paramA-\paramSmoother.dat};
			\addplot [mark=x,Dark2-D] table[x=k,y expr=-2*\thisrow{k}/log10(\thisrow{cheb1_10}),col sep=comma] {data-varcoef-\paramN-\paramE-\paramA-\paramSmoother.dat};
			\addplot [mark=x,Dark2-E] table[x=k,y expr=-2*\thisrow{k}/log10(\thisrow{cheb1_30}),col sep=comma] {data-varcoef-\paramN-\paramE-\paramA-\paramSmoother.dat};
			\addplot [mark=x,Dark2-F] table[x=k,y expr=-2*\thisrow{k}/log10(\thisrow{cheb1_100}),col sep=comma] {data-varcoef-\paramN-\paramE-\paramA-\paramSmoother.dat};
			\addplot [mark=x,Dark2-G] table[x=k,y expr=-2*\thisrow{k}/log10(\thisrow{cheb1_300}),col sep=comma] {data-varcoef-\paramN-\paramE-\paramA-\paramSmoother.dat};
			
			\addplot [mark=o,Dark2-C] table[x=k,y expr=-2*\thisrow{k}/log10(\thisrow{w43}),col sep=comma] {data-varcoef-\paramN-\paramE-\paramA-\paramSmoother.dat};
			\addplot [mark=o,Dark2-D] table[x=k,y expr=-2*\thisrow{k}/log10(\thisrow{w32}),col sep=comma] {data-varcoef-\paramN-\paramE-\paramA-\paramSmoother.dat};
			
		\end{groupplot}
		\coordinate (c3) at ($(c1)!.5!(c2)$);
		\node[below] at (c3 |- current bounding box.south)
		{\pgfplotslegendfromname{leg-varcoef}};	
		
	\end{tikzpicture}
	
	\caption{V-Cycle error contraction factors and smoother steps per digit for the discretized Poisson problem with coefficient jumps using (rescaled) Jacobi for the single-step smoother.}
	\label{fig:nr-varcoef}
\end{figure}

We next consider a problem involving coefficient jumps. We use the same grid as before, 1024 by 1024 bilinear square finite elements in a uniform grid. We now fix the aspect ratio to be 1, and group the elements into square macroelements each consisting of $N = 32$ by $N$ elements, resulting an array of $M=32$ by $M$ macroelements. We set the diffusion coefficient to $1$ in half of the macroelements and to a parameter $a$ in the other half, arranged in a checkerboard pattern. Note that by taking $N$ to be a power of 2, the usual coarse grids resolve all of the macroelement boundaries (the interfaces of coefficient jumps) until at some level 2 by 2 squares of macroelements are finally merged into a single element (with the Galerkin construction effectively making the diffusion coefficient in the merged element a harmonic average). As a result, as can be seen in \cref{fig:nr-varcoef}, the simple multigrid V-cycle with standard bilinear interpolation and damped Jacobi smoothing has a convergence rate that appears to be bounded independently of the problem size (number of macroelements) and the size of the coefficient jump characterized by $a$. Results for intermediate values of $a$ are similar, with convergence rates appearing to asymptote for large $a$ (making $a=10^8$ representative of the very large coefficient jump regime).

\begin{table}[htb]
	\centering
	\begin{tabular}{rrrrrr} \toprule
		& & \multicolumn{4}{c}{$C$} \\ \cmidrule(l){3-6}
		$N$ & $M$ &$a=10$ & $a=10^2$ & $a=10^4$ & $a=10^8$ \\ \midrule
		2 & 256 & 1.77e4 & 2.97e4 & 3.16e4 & 3.16e4 \\
		2 & 512 & 7.10e4 & 1.19e5 & 1.27e5 & 1.27e5 \\ \midrule
		4 & 128 & 8.65e3 & 2.21e4 & 2.51e4 & 2.51e4 \\
		4 & 256 & 3.48e4 & 8.91e4 & 1.01e5 & 1.01e5 \\ \midrule
		8 & 64 & 4.99e3 & 1.89e4 & 2.30e4 & 2.31e4 \\
		8 & 128 & 2.02e4 & 7.69e4 & 9.36e4 & 9.38e4 \\ \midrule
		16 & 32 & 2.93e3 & 1.68e4 & 2.21e4 & 2.22e4 \\
		16 & 64 & 1.20e4 & 6.89e4 & 9.10e4 & 9.13e4 \\ \midrule
		32 & 16 & 1.67e3 & 1.44e4 & 2.09e4 & 2.09e4 \\
		32 & 32 & 6.98e3 & 6.09e4 & 8.80e4 & 8.84e4 \\
		\bottomrule \end{tabular}
	\caption{Approximation constant $C$ for multigrid on the discretized Poisson problem with coefficient jumps.} \label{tbl:jump-C}
\end{table}

Strikingly, for this problem, we see very little improvement of the convergence rates when we increase the polynomial degree $k$, and $k=1$ is the most efficient in terms of minimizing the total number of smoothing steps. How does this conform, e.g., to the bound of \cref{cor:cheb-bound} for the fourth-kind Chebyshev smoother, which asymptotically behaves like $\tfrac{3}{4} C k^{-2}$? As we see in \cref{tbl:jump-C}, the approximation constant $C$ for this problem, measured using a generalized Lanczos iteration with an embedded multigrid solver, is very large and appears to be unbounded with problem size (the number of macroelements $M$ in each direction), though, interestingly, it does appear to be bounded with respect to the parameter $a$. As such, the bounds based on \cref{lem:multilevel-poly-bound} such as \cref{cor:cheb-bound} are uselessly pessimistic. On the other hand, we see that having a bounded approximation constant $C$, while not a necessary condition for bounded multigrid V-cycle convergence, is a useful sufficient condition for the effectiveness of higher degree polynomial smoothers in a V-cycle, in accordance with \cref{lem:multilevel-poly-bound}. It is possible that a different multigrid strategy for this problem (e.g., AMG) could lead to bounded constants $C$, in which case higher degree polynomial smoothers might be effective.

We note that similar results are obtained if $\ell_1$-Jacobi, explained in the next section, is used instead of Jacobi as the single-step smoother. 

\subsection{Chebyshev spaced grid}

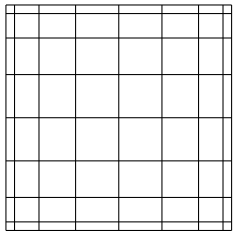
\begin{figure}[tb]
	\centering
  \begin{tikzpicture}[scale=1.5]
		\foreach \y in {-1, -0.9238795325112867, -0.7071067811865475, -0.3826834323650897, 0, 0.3826834323650898, 0.7071067811865476, 0.9238795325112867, 1} {
			\draw (-1, \y) -- (1, \y);
	    }
    
		\foreach \x in {-1, -0.9238795325112867, -0.7071067811865475, -0.3826834323650897, 0, 0.3826834323650898, 0.7071067811865476, 0.9238795325112867, 1} {
	\draw (\x, -1) -- (\x, 1);
}
  \end{tikzpicture}
\caption{$N$ by $N$ macroelement with Chebyshev grid spacing, $N=8$.} \label{fig:chebyshev-macroelement}
\end{figure}

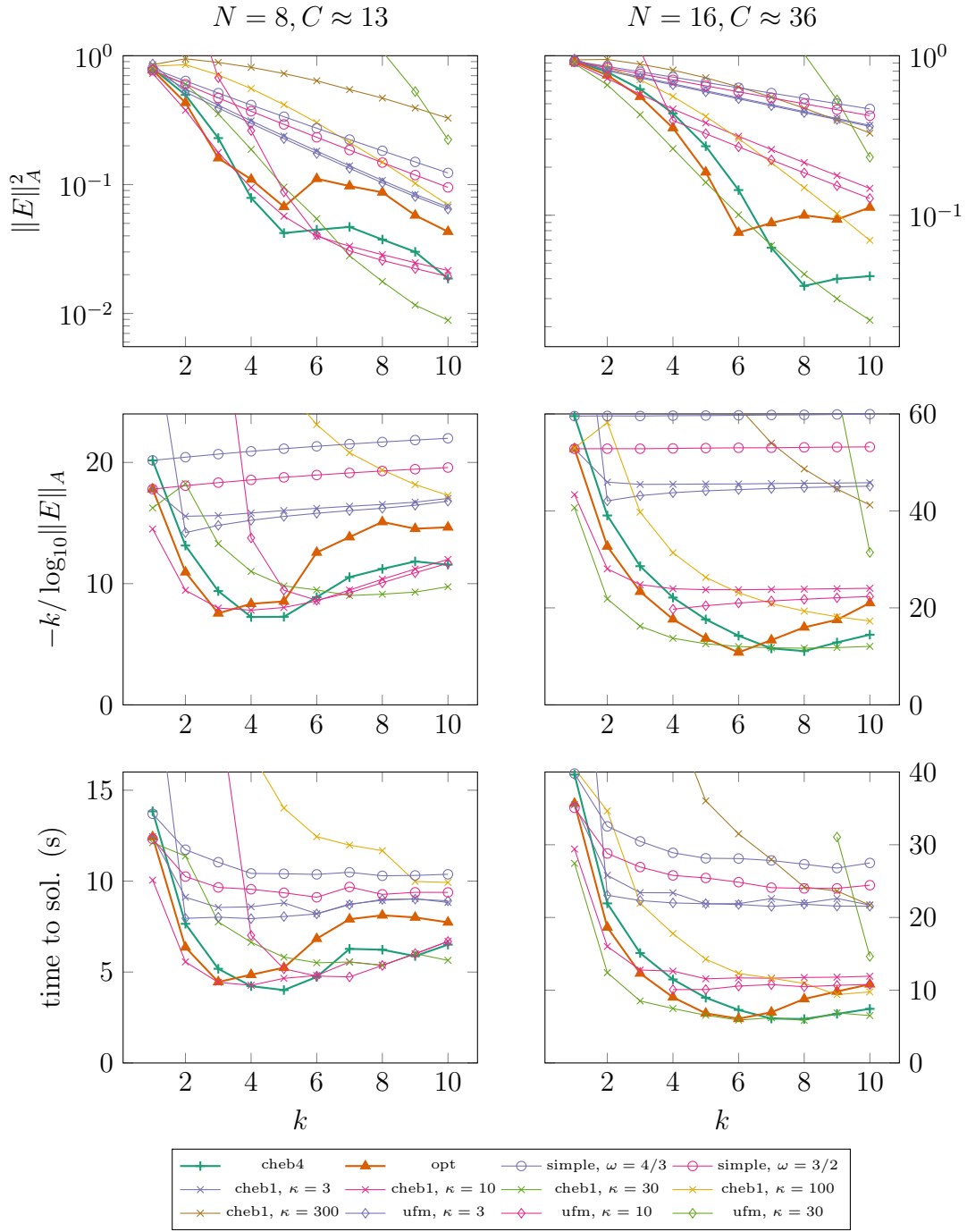
\begin{figure}[ptb]
	\centering
	\begin{tikzpicture}
		\begin{groupplot}[%
			group style={%
				group size=2 by 3,
			},
			width=.5\linewidth
			]
			\nextgroupplot[title={$N=8, C\approx 13$},ymode=log,ymax=1,ylabel={$\norm{E}_A^2$},
			legend style={font=\tiny,at={($(0,0)+(1cm,1cm)$)}, legend columns=4,fill=none,draw=black,anchor=center,align=center},
			legend to name=leg-chebgrid
			]
			\addlegendimage{mark=+,Dark2-A,style=thick,mark options={thick}}; \addlegendentry{cheb4};
			\addlegendimage{mark=triangle*,Dark2-B,style=thick}; \addlegendentry{opt};
			\addlegendimage{mark=o,Dark2-C}; \addlegendentry{simple, $\omega=4/3$};
			\addlegendimage{mark=o,Dark2-D}; \addlegendentry{simple, $\omega=3/2$};
			
			\addlegendimage{mark=x,Dark2-C}; \addlegendentry{cheb1, $\kappa=3$};
			\addlegendimage{mark=x,Dark2-D}; \addlegendentry{cheb1, $\kappa=10$};
			\addlegendimage{mark=x,Dark2-E}; \addlegendentry{cheb1, $\kappa=30$};
			\addlegendimage{mark=x,Dark2-F}; \addlegendentry{cheb1, $\kappa=100$};
			\addlegendimage{mark=x,Dark2-G}; \addlegendentry{cheb1, $\kappa=300$};
			
			\addlegendimage{mark=diamond,Dark2-C}; \addlegendentry{ufm, $\kappa=3$};
			\addlegendimage{mark=diamond,Dark2-D}; \addlegendentry{ufm, $\kappa=10$};
			\addlegendimage{mark=diamond,Dark2-E}; \addlegendentry{ufm, $\kappa=30$};

			\pgfmathsetmacro{\paramN}{8}
			\pgfmathsetmacro{\paramE}{128}
			
			\addplot [mark=+,Dark2-A,style=thick,mark options={thick}] table[x=k,y=cheb4,col sep=comma] {data-chebgrid-\paramN-\paramE-dn-l1jacobi.dat};
			\addplot [mark=triangle*,Dark2-B,style=thick] table[x=k,y=opt,col sep=comma] {data-chebgrid-\paramN-\paramE-dn-l1jacobi.dat};
			\addplot [mark=x,Dark2-C] table[x=k,y=cheb1_3,col sep=comma] {data-chebgrid-\paramN-\paramE-dn-l1jacobi.dat};
			\addplot [mark=x,Dark2-D] table[x=k,y=cheb1_10,col sep=comma] {data-chebgrid-\paramN-\paramE-dn-l1jacobi.dat};
			\addplot [mark=x,Dark2-E] table[x=k,y=cheb1_30,col sep=comma] {data-chebgrid-\paramN-\paramE-dn-l1jacobi.dat};
			\addplot [mark=x,Dark2-F] table[x=k,y=cheb1_100,col sep=comma] {data-chebgrid-\paramN-\paramE-dn-l1jacobi.dat};
			\addplot [mark=x,Dark2-G] table[x=k,y=cheb1_300,col sep=comma] {data-chebgrid-\paramN-\paramE-dn-l1jacobi.dat};
			
			\addplot [mark=diamond,Dark2-C] table[x=k,y=ufm_3,col sep=comma] {data-chebgrid-\paramN-\paramE-dn-l1jacobi.dat};
			\addplot [mark=diamond,Dark2-D] table[x=k,y=ufm_10,col sep=comma] {data-chebgrid-\paramN-\paramE-dn-l1jacobi.dat};
			\addplot [mark=diamond,Dark2-E] table[x=k,y=ufm_30,col sep=comma] {data-chebgrid-\paramN-\paramE-dn-l1jacobi.dat};
			
			\addplot [mark=o,Dark2-C] table[x=k,y=w43,col sep=comma] {data-chebgrid-\paramN-\paramE-dn-l1jacobi.dat};
			\addplot [mark=o,Dark2-D] table[x=k,y=w32,col sep=comma] {data-chebgrid-\paramN-\paramE-dn-l1jacobi.dat};
			
			\coordinate (c1) at (rel axis cs:0,1);
			
			\nextgroupplot[title={$N=16, C\approx 36$}, ymode=log,ymax=1,yticklabel pos=right]
			
			\pgfmathsetmacro{\paramN}{16}
			\pgfmathsetmacro{\paramE}{64}
			
			\addplot [mark=+,Dark2-A,style=thick,mark options={thick}] table[x=k,y=cheb4,col sep=comma] {data-chebgrid-\paramN-\paramE-dn-l1jacobi.dat};
			\addplot [mark=triangle*,Dark2-B,style=thick] table[x=k,y=opt,col sep=comma] {data-chebgrid-\paramN-\paramE-dn-l1jacobi.dat};
			\addplot [mark=x,Dark2-C] table[x=k,y=cheb1_3,col sep=comma] {data-chebgrid-\paramN-\paramE-dn-l1jacobi.dat};
			\addplot [mark=x,Dark2-D] table[x=k,y=cheb1_10,col sep=comma] {data-chebgrid-\paramN-\paramE-dn-l1jacobi.dat};
			\addplot [mark=x,Dark2-E] table[x=k,y=cheb1_30,col sep=comma] {data-chebgrid-\paramN-\paramE-dn-l1jacobi.dat};
			\addplot [mark=x,Dark2-F] table[x=k,y=cheb1_100,col sep=comma] {data-chebgrid-\paramN-\paramE-dn-l1jacobi.dat};
			\addplot [mark=x,Dark2-G] table[x=k,y=cheb1_300,col sep=comma] {data-chebgrid-\paramN-\paramE-dn-l1jacobi.dat};
			
			\addplot [mark=diamond,Dark2-C] table[x=k,y=ufm_3,col sep=comma] {data-chebgrid-\paramN-\paramE-dn-l1jacobi.dat};
			\addplot [mark=diamond,Dark2-D] table[x=k,y=ufm_10,col sep=comma] {data-chebgrid-\paramN-\paramE-dn-l1jacobi.dat};
			\addplot [mark=diamond,Dark2-E] table[x=k,y=ufm_30,col sep=comma] {data-chebgrid-\paramN-\paramE-dn-l1jacobi.dat};
			
			\addplot [mark=o,Dark2-C] table[x=k,y=w43,col sep=comma] {data-chebgrid-\paramN-\paramE-dn-l1jacobi.dat};
			\addplot [mark=o,Dark2-D] table[x=k,y=w32,col sep=comma] {data-chebgrid-\paramN-\paramE-dn-l1jacobi.dat};
			
			\coordinate (c2) at (rel axis cs:1,1);
			
			\nextgroupplot[ylabel={$-k/\log_{10} \norm{E}_A$},ymax=24,ymin=0]
			
			\pgfmathsetmacro{\paramN}{8}
			\pgfmathsetmacro{\paramE}{128}
			
			\addplot [mark=+,Dark2-A,style=thick,mark options={thick}] table[x=k,y expr=-2*\thisrow{k}/log10(\thisrow{cheb4}),col sep=comma] {data-chebgrid-\paramN-\paramE-dn-l1jacobi.dat};
			\addplot [mark=triangle*,Dark2-B,style=thick] table[x=k,y expr=-2*\thisrow{k}/log10(\thisrow{opt}),col sep=comma] {data-chebgrid-\paramN-\paramE-dn-l1jacobi.dat};
			\addplot [mark=x,Dark2-C] table[x=k,y expr=-2*\thisrow{k}/log10(\thisrow{cheb1_3}),col sep=comma] {data-chebgrid-\paramN-\paramE-dn-l1jacobi.dat};
			\addplot [mark=x,Dark2-D] table[x=k,y expr=-2*\thisrow{k}/log10(\thisrow{cheb1_10}),col sep=comma] {data-chebgrid-\paramN-\paramE-dn-l1jacobi.dat};
			\addplot [mark=x,Dark2-E] table[x=k,y expr=-2*\thisrow{k}/log10(\thisrow{cheb1_30}),col sep=comma] {data-chebgrid-\paramN-\paramE-dn-l1jacobi.dat};
			\addplot [mark=x,Dark2-F] table[x=k,y expr=-2*\thisrow{k}/log10(\thisrow{cheb1_100}),col sep=comma] {data-chebgrid-\paramN-\paramE-dn-l1jacobi.dat};
			\addplot [mark=x,Dark2-G] table[x=k,y expr=-2*\thisrow{k}/log10(\thisrow{cheb1_300}),col sep=comma] {data-chebgrid-\paramN-\paramE-dn-l1jacobi.dat};
			
			\addplot [mark=diamond,Dark2-C] table[x=k,y expr=-2*\thisrow{k}/log10(\thisrow{ufm_3}),col sep=comma] {data-chebgrid-\paramN-\paramE-dn-l1jacobi.dat};
			\addplot [mark=diamond,Dark2-D] table[x=k,y expr=-2*\thisrow{k}/log10(\thisrow{ufm_10}),col sep=comma,skip coords between index={0}{2}] {data-chebgrid-\paramN-\paramE-dn-l1jacobi.dat};
			\addplot [mark=diamond,Dark2-E] table[x=k,y expr=-2*\thisrow{k}/log10(\thisrow{ufm_30}),col sep=comma,skip coords between index={0}{8}] {data-chebgrid-\paramN-\paramE-dn-l1jacobi.dat};
			\addplot [mark=o,Dark2-C] table[x=k,y expr=-2*\thisrow{k}/log10(\thisrow{w43}),col sep=comma] {data-chebgrid-\paramN-\paramE-dn-l1jacobi.dat};
			\addplot [mark=o,Dark2-D] table[x=k,y expr=-2*\thisrow{k}/log10(\thisrow{w32}),col sep=comma] {data-chebgrid-\paramN-\paramE-dn-l1jacobi.dat};
			
			\nextgroupplot[ymin=0,ymax=60,yticklabel pos=right]
			
			\pgfmathsetmacro{\paramN}{16}
			\pgfmathsetmacro{\paramE}{64}
			
			\addplot [mark=+,Dark2-A,style=thick,mark options={thick}] table[x=k,y expr=-2*\thisrow{k}/log10(\thisrow{cheb4}),col sep=comma] {data-chebgrid-\paramN-\paramE-dn-l1jacobi.dat};
			\addplot [mark=triangle*,Dark2-B,style=thick] table[x=k,y expr=-2*\thisrow{k}/log10(\thisrow{opt}),col sep=comma] {data-chebgrid-\paramN-\paramE-dn-l1jacobi.dat};
			\addplot [mark=x,Dark2-C] table[x=k,y expr=-2*\thisrow{k}/log10(\thisrow{cheb1_3}),col sep=comma] {data-chebgrid-\paramN-\paramE-dn-l1jacobi.dat};
			\addplot [mark=x,Dark2-D] table[x=k,y expr=-2*\thisrow{k}/log10(\thisrow{cheb1_10}),col sep=comma] {data-chebgrid-\paramN-\paramE-dn-l1jacobi.dat};
			\addplot [mark=x,Dark2-E] table[x=k,y expr=-2*\thisrow{k}/log10(\thisrow{cheb1_30}),col sep=comma] {data-chebgrid-\paramN-\paramE-dn-l1jacobi.dat};
			\addplot [mark=x,Dark2-F] table[x=k,y expr=-2*\thisrow{k}/log10(\thisrow{cheb1_100}),col sep=comma] {data-chebgrid-\paramN-\paramE-dn-l1jacobi.dat};
			\addplot [mark=x,Dark2-G] table[x=k,y expr=-2*\thisrow{k}/log10(\thisrow{cheb1_300}),col sep=comma] {data-chebgrid-\paramN-\paramE-dn-l1jacobi.dat};
			
			\addplot [mark=diamond,Dark2-C] table[x=k,y expr=-2*\thisrow{k}/log10(\thisrow{ufm_3}),col sep=comma] {data-chebgrid-\paramN-\paramE-dn-l1jacobi.dat};
			\addplot [mark=diamond,Dark2-D] table[x=k,y expr=-2*\thisrow{k}/log10(\thisrow{ufm_10}),col sep=comma,skip coords between index={0}{3}] {data-chebgrid-\paramN-\paramE-dn-l1jacobi.dat};
			\addplot [mark=diamond,Dark2-E] table[x=k,y expr=-2*\thisrow{k}/log10(\thisrow{ufm_30}),col sep=comma,skip coords between index={0}{8}] {data-chebgrid-\paramN-\paramE-dn-l1jacobi.dat};
			\addplot [mark=o,Dark2-C] table[x=k,y expr=-2*\thisrow{k}/log10(\thisrow{w43}),col sep=comma] {data-chebgrid-\paramN-\paramE-dn-l1jacobi.dat};
			\addplot [mark=o,Dark2-D] table[x=k,y expr=-2*\thisrow{k}/log10(\thisrow{w32}),col sep=comma] {data-chebgrid-\paramN-\paramE-dn-l1jacobi.dat};
			
			\nextgroupplot[ylabel={time to sol. (s)},xlabel=$k$,ymax=16,ymin=0]
			
			\pgfmathsetmacro{\paramN}{8}
			\pgfmathsetmacro{\paramE}{128}
			\def\paramSmoother{l1jacobi}
			
			\addplot [mark=+,Dark2-A,style=thick,mark options={thick}] table[x=k,y=cheb4,col sep=comma] {data-chebgrid-timing-\paramN-\paramE-dn-\paramSmoother.dat};
			\addplot [mark=triangle*,Dark2-B,style=thick] table[x=k,y=opt,col sep=comma] {data-chebgrid-timing-\paramN-\paramE-dn-\paramSmoother.dat};
			\addplot [mark=x,Dark2-C] table[x=k,y=cheb1_3,col sep=comma] {data-chebgrid-timing-\paramN-\paramE-dn-\paramSmoother.dat};
			\addplot [mark=x,Dark2-D] table[x=k,y=cheb1_10,col sep=comma] {data-chebgrid-timing-\paramN-\paramE-dn-\paramSmoother.dat};
			\addplot [mark=x,Dark2-E] table[x=k,y=cheb1_30,col sep=comma] {data-chebgrid-timing-\paramN-\paramE-dn-\paramSmoother.dat};
			\addplot [mark=x,Dark2-F] table[x=k,y=cheb1_100,col sep=comma] {data-chebgrid-timing-\paramN-\paramE-dn-\paramSmoother.dat};
			\addplot [mark=x,Dark2-G] table[x=k,y=cheb1_300,col sep=comma,col sep=comma,skip coords between index={0}{3}] {data-chebgrid-timing-\paramN-\paramE-dn-\paramSmoother.dat};
			
			\addplot [mark=diamond,Dark2-C] table[x=k,y=ufm_3,col sep=comma] {data-chebgrid-timing-\paramN-\paramE-dn-\paramSmoother.dat};
			\addplot [mark=diamond,Dark2-D] table[x=k,y=ufm_10,col sep=comma,skip coords between index={0}{2}] {data-chebgrid-timing-\paramN-\paramE-dn-\paramSmoother.dat};
			\addplot [mark=diamond,Dark2-E] table[x=k,y=ufm_30,col sep=comma,skip coords between index={0}{8}] {data-chebgrid-timing-\paramN-\paramE-dn-\paramSmoother.dat};
			\addplot [mark=o,Dark2-C] table[x=k,y=w43,col sep=comma] {data-chebgrid-timing-\paramN-\paramE-dn-\paramSmoother.dat};
			\addplot [mark=o,Dark2-D] table[x=k,y=w32,col sep=comma] {data-chebgrid-timing-\paramN-\paramE-dn-\paramSmoother.dat};
			
			\nextgroupplot[ymin=0,ymax=40,xlabel=$k$,yticklabel pos=right]
			
			\pgfmathsetmacro{\paramN}{16}
			\pgfmathsetmacro{\paramE}{64}
			\def\paramSmoother{l1jacobi}
			
			\addplot [mark=+,Dark2-A,style=thick,mark options={thick}] table[x=k,y=cheb4,col sep=comma] {data-chebgrid-timing-\paramN-\paramE-dn-\paramSmoother.dat};
			\addplot [mark=triangle*,Dark2-B,style=thick] table[x=k,y=opt,col sep=comma] {data-chebgrid-timing-\paramN-\paramE-dn-\paramSmoother.dat};
			\addplot [mark=x,Dark2-C] table[x=k,y=cheb1_3,col sep=comma] {data-chebgrid-timing-\paramN-\paramE-dn-\paramSmoother.dat};
			\addplot [mark=x,Dark2-D] table[x=k,y=cheb1_10,col sep=comma] {data-chebgrid-timing-\paramN-\paramE-dn-\paramSmoother.dat};
			\addplot [mark=x,Dark2-E] table[x=k,y=cheb1_30,col sep=comma] {data-chebgrid-timing-\paramN-\paramE-dn-\paramSmoother.dat};
			\addplot [mark=x,Dark2-F] table[x=k,y=cheb1_100,col sep=comma] {data-chebgrid-timing-\paramN-\paramE-dn-\paramSmoother.dat};
			\addplot [mark=x,Dark2-G] table[x=k,y=cheb1_300,col sep=comma,col sep=comma,skip coords between index={0}{3}] {data-chebgrid-timing-\paramN-\paramE-dn-\paramSmoother.dat};
			
			\addplot [mark=diamond,Dark2-C] table[x=k,y=ufm_3,col sep=comma] {data-chebgrid-timing-\paramN-\paramE-dn-\paramSmoother.dat};
			\addplot [mark=diamond,Dark2-D] table[x=k,y=ufm_10,col sep=comma,skip coords between index={0}{3}] {data-chebgrid-timing-\paramN-\paramE-dn-\paramSmoother.dat};
			\addplot [mark=diamond,Dark2-E] table[x=k,y=ufm_30,col sep=comma,skip coords between index={0}{8}] {data-chebgrid-timing-\paramN-\paramE-dn-\paramSmoother.dat};
			\addplot [mark=o,Dark2-C] table[x=k,y=w43,col sep=comma] {data-chebgrid-timing-\paramN-\paramE-dn-\paramSmoother.dat};
			\addplot [mark=o,Dark2-D] table[x=k,y=w32,col sep=comma] {data-chebgrid-timing-\paramN-\paramE-dn-\paramSmoother.dat};
			
		\end{groupplot}
		\coordinate (c3) at ($(c1)!.5!(c2)$);
		\node[below] at (c3 |- current bounding box.south)
		{\pgfplotslegendfromname{leg-chebgrid}};	
		
	\end{tikzpicture}
	
	\caption{V-Cycle error contraction factors and smoother steps per digit for Poisson discretized on a Chebyshev grid using $\ell_1$-Jacobi smoothing.}
	\label{fig:nr-chebgrid}
\end{figure}

For our final example, we again consider Poisson discretized by the FEM on a 1024 by 1024 grid of bilinear rectangular elements. As with our previous example, we group the elements into $N$ by $N$ square macroelements. In this case, instead of altering the diffusion coefficient, we use Chebyshev nodal spacing within each macroelement. That is, element boundaries are located at the Chebyshev nodes $x_j = \cos \tfrac{j}{N} \pi$, $j = 0, \dotsc, N$, affinely mapped to each macroelement. A single macroelement for the case $N=8$ is illustrated in \cref{fig:chebyshev-macroelement}. We take homogeneous Dirichlet boundary conditions for the horizontal boundaries and Neumann for the vertical. We also use $\ell_1$-Jacobi \cite{baker_multigrid_2011} as the single-step smoother. This diagonal smoother has components $B_{ii}^{-1} = \sum_{j=1}^n |A_{ij}|$, and has the advantage of satisfying $\rho(BA) \le 1$ by construction (as can be seen from the Gershgorin circle theorem), so that it need not be rescaled.

We note that this problem is of practical interest as it serves as a proxy for a low-order based preconditioner for a high-order element problem, considered, e.g., in \cite{heys_algebraic_2005}. Also note that, while the high aspect ratio elements of our first model problem could be handled by a semi-coarsening strategy, such an approach is not viable here. Approaches based on algebraic multigrid can result in hierarchies with small approximation constants $C$ at each level, but typically at the expense of increased operator complexity. As a result, geometric multigrid can be competitive, depending on specifics of the problem. The use of polynomial acceleration of the smoother iteration, as we will see, can further increase the competitiveness of a geometric multigrid approach.

Asymptotic convergence rates $\norm{E}_A^2$ for the symmetric V-cycle as well as asymptotic smoother steps per decimal digit $-k/\log_{10} \norm{E}_A$ are shown in \cref{fig:nr-chebgrid} for the cases $N=8$ and $N=16$. Overall, results are quite similar to those for our first model problem. We have also included plots of time to solution---the time to reduce the 2-norm of the same random initial residual by the factor $10^{10}$ using a simple Matlab implementation of the symmetric multigrid V-cycle running on a desktop. We include the plot mainly to prove the point that the metric $-k/\log_{10} \norm{E}_A$ is indeed indicative of real performance. There are two principal differences. For one, the total number of smoothing steps in an actual run is necessarily a multiple of $2k$, leading to a ``quantization'' effect. More significantly, the time to solution includes the time taken by the prolongation and restriction steps, which scale with the number of V-cycles, thus typically decreasing with $k$. In particular, this causes even the simple smoothers to benefit from increasing $k$. More generally, the effect is that the $-k/\log_{10} \norm{E}_A$ metric may underpredict the optimal $k$. The time to solution plots also clearly demonstrate the benefit of using polynomial smoothers for this problem, and in going to higher polynomial degrees for the more difficult case, a trend that would continue if we increase $N$. Moreover, we see that the Chebyshev polynomials of the first kind with a well chosen parameter $\kappa$, the Chebyshev polynomials of the fourth kind, and the optimized polynomials are all competitive in terms of achieving the minimal time to solution.

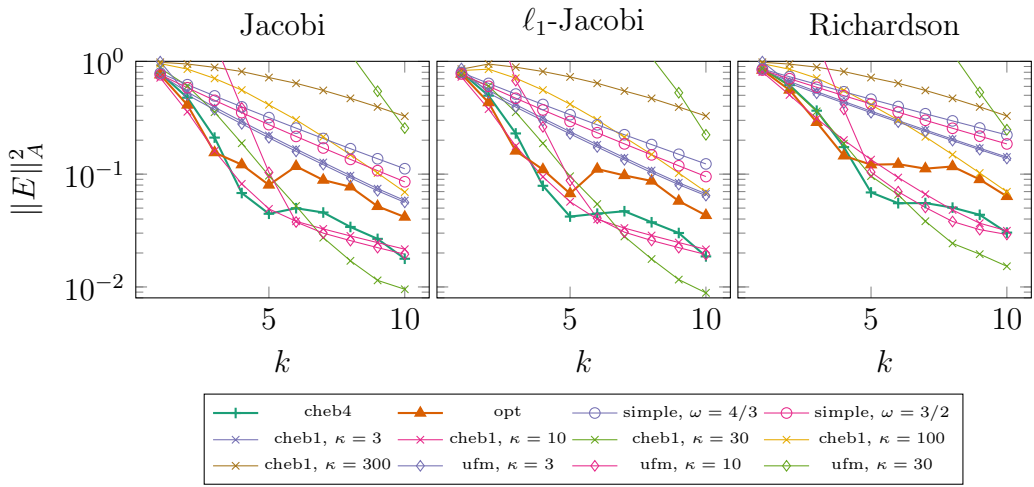
\begin{figure}[ptb]
	\centering
	\begin{tikzpicture}
		\begin{groupplot}[%
			group style={%
				group size=3 by 1,
				horizontal sep=.1cm,
			},
			width=.4\linewidth,
			ymax=1,
			ymin=.008
			]
			\nextgroupplot[title={Jacobi},ymode=log,ymax=1,ylabel={$\norm{E}_A^2$},xlabel=$k$,
			legend style={font=\tiny,at={($(0,0)+(1cm,1cm)$)}, legend columns=4,fill=none,draw=black,anchor=center,align=center},
			legend to name=leg-chebgrid3
			]
			
			\addlegendimage{mark=+,Dark2-A,style=thick,mark options={thick}}; \addlegendentry{cheb4};
			\addlegendimage{mark=triangle*,Dark2-B,style=thick}; \addlegendentry{opt};
			\addlegendimage{mark=o,Dark2-C}; \addlegendentry{simple, $\omega=4/3$};
			\addlegendimage{mark=o,Dark2-D}; \addlegendentry{simple, $\omega=3/2$};
			
			\addlegendimage{mark=x,Dark2-C}; \addlegendentry{cheb1, $\kappa=3$};
			\addlegendimage{mark=x,Dark2-D}; \addlegendentry{cheb1, $\kappa=10$};
			\addlegendimage{mark=x,Dark2-E}; \addlegendentry{cheb1, $\kappa=30$};
			\addlegendimage{mark=x,Dark2-F}; \addlegendentry{cheb1, $\kappa=100$};
			\addlegendimage{mark=x,Dark2-G}; \addlegendentry{cheb1, $\kappa=300$};
			
			\addlegendimage{mark=diamond,Dark2-C}; \addlegendentry{ufm, $\kappa=3$};
			\addlegendimage{mark=diamond,Dark2-D}; \addlegendentry{ufm, $\kappa=10$};
			\addlegendimage{mark=diamond,Dark2-E}; \addlegendentry{ufm, $\kappa=30$};
			
			\pgfmathsetmacro{\paramN}{8}
			\pgfmathsetmacro{\paramE}{128}
			\def\paramSmoother{jacobi}
			
			\addplot [mark=+,Dark2-A,style=thick,mark options={thick}] table[x=k,y=cheb4,col sep=comma] {data-chebgrid-\paramN-\paramE-dn-\paramSmoother.dat};
			\addplot [mark=triangle*,Dark2-B,style=thick] table[x=k,y=opt,col sep=comma] {data-chebgrid-\paramN-\paramE-dn-\paramSmoother.dat};
			\addplot [mark=x,Dark2-C] table[x=k,y=cheb1_3,col sep=comma] {data-chebgrid-\paramN-\paramE-dn-\paramSmoother.dat};
			\addplot [mark=x,Dark2-D] table[x=k,y=cheb1_10,col sep=comma] {data-chebgrid-\paramN-\paramE-dn-\paramSmoother.dat};
			\addplot [mark=x,Dark2-E] table[x=k,y=cheb1_30,col sep=comma] {data-chebgrid-\paramN-\paramE-dn-\paramSmoother.dat};
			\addplot [mark=x,Dark2-F] table[x=k,y=cheb1_100,col sep=comma] {data-chebgrid-\paramN-\paramE-dn-\paramSmoother.dat};
			\addplot [mark=x,Dark2-G] table[x=k,y=cheb1_300,col sep=comma] {data-chebgrid-\paramN-\paramE-dn-\paramSmoother.dat};
			
			\addplot [mark=diamond,Dark2-C] table[x=k,y=ufm_3,col sep=comma] {data-chebgrid-\paramN-\paramE-dn-\paramSmoother.dat};
			\addplot [mark=diamond,Dark2-D] table[x=k,y=ufm_10,col sep=comma] {data-chebgrid-\paramN-\paramE-dn-\paramSmoother.dat};
			\addplot [mark=diamond,Dark2-E] table[x=k,y=ufm_30,col sep=comma] {data-chebgrid-\paramN-\paramE-dn-\paramSmoother.dat};
			\addplot [mark=o,Dark2-C] table[x=k,y=w43,col sep=comma] {data-chebgrid-\paramN-\paramE-dn-\paramSmoother.dat};
			\addplot [mark=o,Dark2-D] table[x=k,y=w32,col sep=comma] {data-chebgrid-\paramN-\paramE-dn-\paramSmoother.dat};
			
			\coordinate (c1) at (rel axis cs:0,1);
			
			\nextgroupplot[title={$\ell_1$-Jacobi}, ymode=log,ymax=1,yticklabels={,,},xlabel=$k$]
			
			\pgfmathsetmacro{\paramN}{8}
			\pgfmathsetmacro{\paramE}{128}
			\def\paramSmoother{l1jacobi}
			
			\addplot [mark=+,Dark2-A,style=thick,mark options={thick}] table[x=k,y=cheb4,col sep=comma] {data-chebgrid-\paramN-\paramE-dn-\paramSmoother.dat};
			\addplot [mark=triangle*,Dark2-B,style=thick] table[x=k,y=opt,col sep=comma] {data-chebgrid-\paramN-\paramE-dn-\paramSmoother.dat};
			\addplot [mark=x,Dark2-C] table[x=k,y=cheb1_3,col sep=comma] {data-chebgrid-\paramN-\paramE-dn-\paramSmoother.dat};
			\addplot [mark=x,Dark2-D] table[x=k,y=cheb1_10,col sep=comma] {data-chebgrid-\paramN-\paramE-dn-\paramSmoother.dat};
			\addplot [mark=x,Dark2-E] table[x=k,y=cheb1_30,col sep=comma] {data-chebgrid-\paramN-\paramE-dn-\paramSmoother.dat};
			\addplot [mark=x,Dark2-F] table[x=k,y=cheb1_100,col sep=comma] {data-chebgrid-\paramN-\paramE-dn-\paramSmoother.dat};
			\addplot [mark=x,Dark2-G] table[x=k,y=cheb1_300,col sep=comma] {data-chebgrid-\paramN-\paramE-dn-\paramSmoother.dat};
			
			\addplot [mark=diamond,Dark2-C] table[x=k,y=ufm_3,col sep=comma] {data-chebgrid-\paramN-\paramE-dn-\paramSmoother.dat};
			\addplot [mark=diamond,Dark2-D] table[x=k,y=ufm_10,col sep=comma] {data-chebgrid-\paramN-\paramE-dn-\paramSmoother.dat};
			\addplot [mark=diamond,Dark2-E] table[x=k,y=ufm_30,col sep=comma] {data-chebgrid-\paramN-\paramE-dn-\paramSmoother.dat};
			\addplot [mark=o,Dark2-C] table[x=k,y=w43,col sep=comma] {data-chebgrid-\paramN-\paramE-dn-\paramSmoother.dat};
			\addplot [mark=o,Dark2-D] table[x=k,y=w32,col sep=comma] {data-chebgrid-\paramN-\paramE-dn-\paramSmoother.dat};
			
			\nextgroupplot[title={Richardson}, ymode=log,ymax=1,yticklabels={,,},xlabel=$k$]
			
			\pgfmathsetmacro{\paramN}{8}
			\pgfmathsetmacro{\paramE}{128}
			\def\paramSmoother{richardson}
			
			\addplot [mark=+,Dark2-A,style=thick,mark options={thick}] table[x=k,y=cheb4,col sep=comma] {data-chebgrid-\paramN-\paramE-dn-\paramSmoother.dat};
			\addplot [mark=triangle*,Dark2-B,style=thick] table[x=k,y=opt,col sep=comma] {data-chebgrid-\paramN-\paramE-dn-\paramSmoother.dat};
			\addplot [mark=x,Dark2-C] table[x=k,y=cheb1_3,col sep=comma] {data-chebgrid-\paramN-\paramE-dn-\paramSmoother.dat};
			\addplot [mark=x,Dark2-D] table[x=k,y=cheb1_10,col sep=comma] {data-chebgrid-\paramN-\paramE-dn-\paramSmoother.dat};
			\addplot [mark=x,Dark2-E] table[x=k,y=cheb1_30,col sep=comma] {data-chebgrid-\paramN-\paramE-dn-\paramSmoother.dat};
			\addplot [mark=x,Dark2-F] table[x=k,y=cheb1_100,col sep=comma] {data-chebgrid-\paramN-\paramE-dn-\paramSmoother.dat};
			\addplot [mark=x,Dark2-G] table[x=k,y=cheb1_300,col sep=comma] {data-chebgrid-\paramN-\paramE-dn-\paramSmoother.dat};
			
			\addplot [mark=diamond,Dark2-C] table[x=k,y=ufm_3,col sep=comma] {data-chebgrid-\paramN-\paramE-dn-\paramSmoother.dat};
			\addplot [mark=diamond,Dark2-D] table[x=k,y=ufm_10,col sep=comma] {data-chebgrid-\paramN-\paramE-dn-\paramSmoother.dat};
			\addplot [mark=diamond,Dark2-E] table[x=k,y=ufm_30,col sep=comma] {data-chebgrid-\paramN-\paramE-dn-\paramSmoother.dat};
			\addplot [mark=o,Dark2-C] table[x=k,y=w43,col sep=comma] {data-chebgrid-\paramN-\paramE-dn-\paramSmoother.dat};
			\addplot [mark=o,Dark2-D] table[x=k,y=w32,col sep=comma] {data-chebgrid-\paramN-\paramE-dn-\paramSmoother.dat};
			
			\coordinate (c2) at (rel axis cs:1,1);
			
		\end{groupplot}
		\coordinate (c3) at ($(c1)!.5!(c2)$);
		\node[below] at (c3 |- current bounding box.south)
		{\pgfplotslegendfromname{leg-chebgrid3}};	
		
	\end{tikzpicture}
	
	\caption{V-Cycle error contraction factors for different single-step smoothers for Poisson discretized on a Chebyhev grid with $N=8$.}
	\label{fig:nr-chebgrid3}
\end{figure}

In \cref{fig:nr-chebgrid3}, we compare the convergence rates when using three different diagonal single-step smoothers. Richardson performs somewhat worse than the other two, while Jacobi and $\ell_1$-Jacobi perform comparably, neither one being definitively better in all cases. As mentioned, $\ell_1$-Jacobi has the advantage of not needing to estimate the largest eigenvalue for rescaling. In particular, when combined with the fourth-kind Chebyshev or optimized iterations, the result is a truly parameter-free method.

\section{Conclusion}
The fourth-kind Cheybshev iteration \eqref{eq:cheb-iter} is a very simple way to accelerate the simple smoother iteration $(I-\omega BA)^k$ in a multigrid V-cycle for an SPD system, provided the smoother $B$ is also SPD. All that is required is an estimate of the largest eigenvalue of $BA$, $\rho(BA)$, which can be obtained from a few iterations of conjugate gradient or a generalized Lanczos iteration. No further parameters are involved. In particular, compared to the usual  (first-kind) Chebyshev semi-iteration, the fourth-kind iteration avoids the need to pick a minimum eigenvalue determining the range of eigenvalues to target. The use of a single-step smoother like $\ell_1$-Jacobi, which satisfies $\rho(BA) \le 1$ by construction, can further eliminate even the need to estimate the largest eigenvalue.

We have given to our knowledge the first multilevel V-cycle bound for general (SPD) polynomial smoothers, isolating the polynomial from the single-step smoother. The asymptotic convergence rate of the symmetric V-cycle is bounded by the largest value across levels of
\[ \frac{C}{C + \gamma^{-1}}, \quad \gamma \isdef \sup_{0 < \lambda \le 1} \frac{\lambda p(\lambda)^2}{1-p(\lambda)^2}, \]
where $p$ is the polynomial, assumed to correspond to a convergent iteration with $p(0)=1$, and where $C$ is the approximation property constant
\[C = \norm{A^{-1} - P A_c^{-1} P^*}_{A, B}^2 = \sup_{u \in \ran(P)^{\perp_A} \setminus \{0\}} \frac{\norm{u}_{B^{-1}}^2}{\norm{u}_A^2} \ge 1, \]
involving the single-step smoother $B$, assumed to be scaled such that $\rho(BA) \le 1$.

For the fourth-kind Chebyshev iteration, we have $\gamma^{-1} = \frac{4}{3}k(k+1)$ where $k$ is the number of smoothing steps (the degree of the polynomial), which is a significant improvement over $\gamma^{-1} = 2 \omega k$ for the simple damped iteration, with damping parameter $\omega$ not too large. Moreover, it is possible to numerically determine the polynomials that minimize $\gamma$, which appear to give $\gamma^{-1} \ge \frac{4}{\pi^2} (2k+1)^2 - \frac{2}{3}$, giving roughly an 18\% improvement in the convergence bound asymptotically as $k \to \infty$. The optimized polynomials can be implemented as a slight variation of the fourth-kind Chebyshev iteration, using a small set of coefficients computed and tabulated off-line.

In all our numerical experiments, whenever multiple smoothing steps were used, it was always the case that some polynomial performed better than the simple smoother iteration. Increasing smoothing steps was ineffective for our example with coefficient jumps, which featured an unbounded approximation constant $C$ while still attaining a V-cycle convergence rate bounded below 1. Our theory indicates that a bounded constant $C$ is a sufficient condition for polynomial methods to effectively improve the convergence rate as the number of smoothing steps increases. Moreover, we saw that with polynomial methods, the improved convergence with more smoothing steps can more than make up for the increased cost, with the optimal number of smoothing steps typically increasing with the approximation constant $C$.

In our experiments, the standard first-kind Chebyshev iteration was often able to perform somewhat better than the fourth-kind Chebyshev iteration and its optimized variant. However, this depended on a good choice of the parameter $\kappa$ determining the range of eigenvalues targeted, which we saw depended both on the problem and on the number of smoothing steps $k$, typically increasing both with $C$ and $k$. An optimal parameter in one context could perform quite poorly in another. In contrast, the two parameter-free methods tended to stay within ``striking distance'' of the best method, never performing very poorly. Moreover, they remained especially competitive, at least for the problems we considered, when considering only the most efficient number of smoothing steps $k$. In practice, if it is possible to spend time tuning (by amortizing the cost across solving many right-hand-sides), we would recommend trying all options, tuning the parameters involved, and picking the fastest method. Otherwise, we would argue that the fourth-kind Chebyshev iteration or its optimized variant make sensible defaults.

\printbibliography

\begin{appendices}
	
\crefalias{section}{appendix}

\section{Fourth-kind Chebyshev V-cycle Bound}
\label{sec:cheb-bound-proof}

The V-cycle convergence bound of \Cref{cor:cheb-bound} for the 4th kind Chebyshev polynomial smoothers follows from the fact that the scaled and shifted Chebyshev polynomails of the 4th kind
\[ p_k(\lambda) = \frac{1}{2k+1} W_k(1 - 2\lambda) \]
satisfy
\[ \sup_{0 < \lambda \le 1} \frac{\lambda \, p_k(\lambda)^2}{1 - p_k(\lambda)^2} = \frac{1}{\tfrac{4}{3}k(k+1)}, \]
which we prove below.

\begin{proof}[Proof of \cref{cor:cheb-bound}]
	If we let $n = 2k+1 \ge 3$ and $\phi=\tfrac{1}{2}\theta$ where $\cos \theta = 1 - 2 \lambda$, then we can write
	\[ p_k(\lambda) = \frac{1}{n} \frac{\sin n\phi}{\sin \phi} . \]
	For $\phi$ an integer multiple of $\tfrac{\pi}{n}$, $p_k(\lambda) = 0$. Otherwise,
	\[
	\frac{\lambda p_k(\lambda)^2}{1-p_k(\lambda)^2} = \frac{1}{f_n(\phi)}, \quad f_n(\phi) \isdef n^2 \csc^2 n \phi - \csc^2 \phi
	\]
	and it suffices to show
	\[
	\inf_{0<\phi \le \tfrac{\pi}{2}, \, \phi \notin \tfrac{\pi}{n} \mathbb{Z}} f_n(\phi) = \tfrac{4}{3} k (k+1) = \tfrac{1}{3}(n^2 - 1) .
	\]
	
	For $\tfrac{\pi}{n} < \phi \le \tfrac{\pi}{2}$, since $\sin \phi \ge \frac{2}{\pi} \phi \ge \frac{2}{n}$,
	\[ f_n(\phi) \ge n^2 - (\tfrac{n}{2})^2
	= \tfrac{3}{4} n^2 \ge \tfrac{1}{3}(n^2 - 1) . \]
	So we can assume $0 < \phi < \tfrac{\pi}{n}$. By differentiating the classical partial fraction expansion of cotangent,
	\[ \cot z = \frac{1}{z} + 2 z \sum_{j=1}^{\infty} \frac{1}{z^2-j^2 \pi^2}, \]
	we see that
	\[ \csc^2 z = \frac{1}{z^2} + 2 \sum_{j=1}^{\infty} \frac{z^2 + j^2 \pi^2}{(z^2-j^2 \pi^2)^2} \]
	so
	\[ \tfrac{1}{2} f_n(\phi) = 
	\sum_{j=1}^{\infty} \frac{\phi^2 + j^2 (\tfrac{\pi}{n})^2}{(\phi^2-j^2 (\tfrac{\pi}{n})^2)^2}
	- \sum_{j=1}^{\infty} \frac{\phi^2 + j^2 \pi^2}{(\phi^2-j^2 \pi^2)^2}.
	\]
	Since
	\[ \lim_{\phi \to 0} \tfrac{1}{2}f_n(\phi)
	=	\frac{n^2-1}{\pi^2} \sum_{j=1}^{\infty} \frac{1}{j^2} =
	\tfrac{1}{6}(n^2-1),
	\]
	it suffices to show that $f_n(\phi)$ is increasing on the interval $0 < \phi < \tfrac{\pi}{n}$. Differentiating term by term, we find
	\[ \tfrac{1}{2} f_n'(\phi) = 2 \phi \sum_{j=1}^{\infty} a_{n,j}(\phi) \]
	where
	\[a_{n,j}(\phi) \isdef \frac{3 j^2 (\tfrac{\pi}{n})^2 + \phi^2}{(j^2(\tfrac{\pi}{n})^2 - \phi^2)^3} - \frac{3 j^2 \pi^2 + \phi^2}{(j^2\pi^2 - \phi^2)^3} . \]
	If each term $a_{n,j}(\phi)$ in the convergent sum is nonnegative, it will follow that $f_n'(\phi) > 0$. We can write
	\[a_{n,j}(\phi) = 
	\frac{j^2(n^2-1)}{n^6 (j^2(\tfrac{\pi}{n})^2 - \phi^2)^3(j^2\pi^2 - \phi^2)^3} b_{n,j}(\phi) \]
	where
	\[\begin{split} b_{n,j}(\phi) \isdef {} & 6n^4 \pi^2 \phi^6 - 3n^2(n^2+1)j^2 \pi^4 \phi^4 + \\ & j^4(n^4-8n^2+1)\pi^6 \phi^2 + 3 (n^2+1) j^6 \pi^8 . \end{split} \]
	We can further write $b_{n,j}(\phi)$ as
	\[ b_{n,j}(\phi) = \left[ \sqrt{6} n^2\pi\phi^3 - \sqrt{\tfrac{3}{8}}(n^2+1) j^2 \pi^3 \phi \right]^2 + j^4 \pi^6 c_{n,j}(\phi) \]
	where
	\[ c_{n,j}(\phi) \isdef \tfrac{5}{8}(n^4-14n^2+1) \phi^2 +  3(n^2+1)j^2 \pi^2 . \]
	Since $n^4-14n^2+1 > 0$ for $n \ge 4$, it follows that $c_{n,j}(\phi) > 0$ for $n \ge 4$. On the other hand, for $n=3$,
	\[ c_{3,j}(\phi) = 30 j^2 \pi^2 - \tfrac{55}{2} \phi^2 >
	30 \pi^2 - \tfrac{55}{2} (\tfrac{\pi}{3})^2 = \tfrac{485 \pi^2}{18} > 0
	\] 
	since $j \ge 1$ and $0 < \phi < \tfrac{\pi}{3}$. So, in either case, $c_{n,j}(\phi) > 0$, so that $b_{n,j}(\phi) > 0$, whereby $a_{n,j}(\phi) > 0$, and thus $f_n(\phi)$ is increasing for $0 < \phi < \tfrac{\pi}{n}$.
\end{proof}

\section{Computing Optimal Polynomials} \label{sec:app-opt-poly}
The optimal polynomials for \cref{lem:multilevel-poly-bound} can be found numerically by applying Newton's method to a system of nonlinear equations expressing the equioscillation property. Let us represent the polynomial $p$ by its $k$ roots $r_1, \dotsc, r_k$,
\[ p = \prod_{i=1}^k \left(1-\frac{x}{r_i}\right). \]
Note that $p$ satisfies the required condition $p(0) = 1$. We seek to minimize $\gamma = \sup_{0 < x \le 1} f^2$ where
\[ f \isdef w^{1/2} p, \quad w \isdef \frac{x}{1-p^2}, \quad x > 0, \]
and where we take $w(0) \isdef [-2p'(0)]^{-1}$ so that $w$ and $f$ are continuous at $x=0$.
For the optimal $p$, $f$ will equioscillate,
\begin{equation} \label{eq:f-equioscillate} f(0) = |f(x_1)| = \dotsb = |f(x_{k-1})| = |f(1)| \end{equation}
where $x_i$ are the $k-1$ extrema of $f$, $f'(x_i) = 0$, located between the $k$ roots $r_i$ of $p$. For convenience,
take
\[ x_k \isdef 1 \]
and let $F:\mathbb{R}^k \to \mathbb{R}^k$ be the function of the vector of roots $r = \begin{bmatrix} r_1 & \hdots & r_k \end{bmatrix}^T$ with components
\[ F_i = f(0) - |f(x_i)| . \]
Then \eqref{eq:f-equioscillate} is equivalent to $F(r)=0$ and the optimal $r$ can be found with Newton's method. A natural initial guess for $r$ is the vector of roots of the shifted Chebyshev polynomial of the fourth kind,
\[
  r_i^{(0)} = \tfrac{1}{2}-\tfrac{1}{2}\cos \frac{i}{k+\tfrac{1}{2}} \pi, \quad i=1,\dotsc,k .
\]
A second, nested, Newton iteration can be used to compute the extrema $x_1, \dotsc, x_{k-1}$ of $f$. For the first outer iteration, a good initial guess for the extrema are those of the initial $p$,
\[
  x_i^{(0)} = \tfrac{1}{2}-\tfrac{1}{2}\cos \frac{i+\tfrac{1}{2}}{k+\tfrac{1}{2}} \pi, \quad i=1,\dotsc,k-1 .
\] 
Subsequent iterations can use the extrema from the previous iteration for the initial guess.

The total logarithmic derivative of $p$ is
\[ \frac{dp}{p} = \sum_{i=1}^k \frac{dx - \tfrac{x}{r_i} dr_i}{x - r_i} = q\,dx - \sum_{i=1}^k \frac{x}{x - r_i} \frac{ dr_i}{r_i}, \]
where
\[ q(x) \isdef \frac{p'(x)}{p(x)} = \sum_{i=1}^k \frac{1}{x-r_i}, \]
and the total logarithmic derivative of $f$ is
\begin{equation}\label{eq:f-total-log-deriv} \frac{df}{f} = \frac{dx}{2x} + \frac{w}{x} \frac{dp}{p} = (\tfrac{1}{2} + w q) \frac{dx}{x} - \sum_{i=1}^k \frac{w}{x-r_i} \frac{dr_i}{r_i} . \end{equation}
It follows that the extrema $x_i$ of $f$ are determined by the condition $\tfrac{1}{2} + w q = 0$, or equivalently by $g(x_i) = 0$ where
\[ g(x) \isdef \frac{x}{2 w(x)} + x q(x) = \frac{1-p(x)^2}{2} + x q(x) . \]
The negative derivative of $g(x)$ is
\[ -g'(x) = \sum_{i=1}^k \frac{1}{x-r_i} \left[ p(x)^2 + \frac{r_i}{x-r_i} \right], \]
giving the Newton step to update the extrema, $x_i \gets x_i - g(x_i) / g'(x_i)$.

Suppose $x$ is constrained to equal $x_i$, i.e., to be an extremum of $f$ when $i<k$ or else fixed at $1$ when $i=k$. In the first case $\tfrac{1}{2}+wq=0$ and in the second case $dx = 0$, so that in both cases \eqref{eq:f-total-log-deriv} reduces to
\[ \frac{df}{f} = - \sum_{j=1}^k \frac{w}{x-r_j} \frac{dr_j}{r_j}, \]
from which we can read off $\partial f(x_i)/\partial r_j$. Combining with $\partial f(0) / \partial r_j = f(0)^3 / r_j^2$, which can be derived from $f(0) = [-2q(0)]^{-1/2}$, we find that the components of the Jacobian $J$ of $F$ are
\[ J_{ij} = \frac{f(0)^3}{r_j^2} + \frac{w(x_i) |f(x_i)|}{r_j \, (x_i - r_j)}. \]
The Newton step for updating the vector of roots is then just $r \gets r + \Delta r$ where $J \Delta r = -F$. A full Matlab implementation of the algorithm is shown in \cref{fig:opt_roots.m}. Empirically, the initial guesses suggested above are good enough for convergence of the pair of Newton iterations, at least up to $k = 200$.

\lstset{
	style              = Matlab-editor,
	basicstyle         = \small\ttfamily,
	escapechar         = ",
	mlshowsectionrules = true,
}

\begin{figure}[ptb]
\begin{lstlisting}
function r = opt_roots(k)
  function [w,f,g,ngp] = vars(r, x)
    p = prod(1 - x' ./ r, 1)';
    w = x ./ (1 - p.^2);
    f = sqrt(w) .* p;
    q = sum(1 ./ (x' - r), 1)';
    g = x .* (1./(2*w) + q);
    ngp = sum((p'.^2 + r./(x'-r))./(x'-r), 1)';
  end

  r = .5 - .5*cos([1:k]'/(k+.5) * pi);
  x = .5 - .5*cos((.5+[1:k-1]')/(k+.5) * pi);

  dr = r; drsize = 1;
  while drsize > 128*eps
    dx = x; dxsize = 1;
    while dxsize > 128*eps
      dxsize = norm(dx, inf);
      [~,~,g,ngp] = vars(r,x);
      dx = g ./ ngp; x = x + dx;
    end
    x1 = [x;1]; [w,f,~,~] = vars(r,x1);
    f0 = sqrt(.5/sum(1./r));
    J = f0^3./r'.^2 + w.*abs(f)./(r'.*(x1-r'));
    drsize = norm(dr, inf);
    dr = -J \ (f0 - abs(f)); r = r + dr;
  end
end
\end{lstlisting}
	\caption{A Matlab implementation of Newton's method for finding the roots of the optimal polynomials for \cref{lem:multilevel-poly-bound}.}\label{fig:opt_roots.m}
\end{figure}

Once the roots of the optimal $p_k(\lambda)$ have been determined, we can compute the coefficients $\alpha_i$ of the expansion
\begin{equation}\label{eq:p-exp} p_k(\lambda) = \sum_{i=0}^k \alpha_i W_i(1-2\lambda) \end{equation}
of $p_k$ in the fourth-kind Chebyshev polynomial basis, so that
\[ \beta_0 = 1, \quad \beta_{i+1} = \beta_i - (2i+1) \alpha_i \]
then gives the coefficients $\beta_i$ for implementing $p_k$ in the iteration \eqref{eq:opt-iter}. By combining the orthogonality relation for the fourth-kind Chebyshev polynomials \cite{desmarais_tables_1995}
\[ \frac{1}{\pi} \int_{-1}^1 \sqrt{\frac{1-x}{1+x}} W_i(x)W_j(x) \; dx = \delta_{ij} \]
with the associated Gaussian quadrature rule \cite{desmarais_tables_1995},
\[ \frac{1}{\pi} \int_{-1}^1 \sqrt{\frac{1-x}{1+x}} p(x) \; dx = \sum_{i=1}^k \frac{1-x_i}{k+\tfrac{1}{2}} p(x_i), \quad x_i = \cos \frac{i}{k+\tfrac{1}{2}} \pi, \]
which holds for any polynomial $p(x)$ of degree at most $2k-1$, it follows that if \eqref{eq:p-exp} holds then
\[ \alpha_j = \sum_{i=1}^k \frac{1-x_i}{k+\tfrac{1}{2}} W_j(x_i) p\left(\frac{1-x_i}{2}\right) \quad\text{for $j \le k-1$.}\]
Note that $\alpha_k$ is not needed to compute $\beta_k$. A Matlab implementation of the computation of the coefficients $\beta_i$ is shown in \cref{fig:opt_coef.m}.

\begin{figure}[ptb]
\begin{lstlisting}
function beta = opt_coef(k)
  x = cos([1:k]'*pi/(k+.5));  % Quadrature nodes
  w = (1-x)/(k+.5);  % Quadrature weights

  % 4th kind Chebyshev polynomials W evaluated at x
  W = zeros(k, k);
  W(:, 1) = 1;
  if k >= 2; W(:, 2) = 2*x+1; end
  for i=3:k
    W(:,i) = (2*x) .* W(:,i-1) - W(:,i-2);
  end

  r = opt_roots(k);  % Roots of optimal polynomial p
  lambda = (1-x)/2;  % Nodes transformed to [0, 1]
  p = prod(1-lambda'./r, 1)';  % p(lambda)
  alpha = W' * (w .* p);
  beta = 1 - cumsum((2*[0:k-1]'+1) .* alpha, 1);
end
\end{lstlisting}
	\caption{Matlab code for computing the coefficients $\beta_i$ for the optimized iteration \eqref{eq:opt-iter}.}\label{fig:opt_coef.m}
\end{figure}

\end{appendices}

\end{document}